\documentclass[10pt,twoside]{amsart}
\usepackage{amssymb}
\usepackage{color}
\usepackage{float}
\usepackage[all,cmtip]{xy}
\usepackage{graphicx}
\usepackage[colorlinks,linkcolor=blue]{hyperref}
\usepackage{cancel}
\usepackage{comment}
\usepackage{pstricks}
\usepackage{mathabx}
\usepackage{ifthen, pgfplots, caption}
\usepackage{pgf,tikz}
\usetikzlibrary{arrows}
\usepackage{a4wide}

\pgfplotsset{compat=1.13}

\newcommand{\abs}[1]{\lvert#1\rvert}

\usetikzlibrary{decorations.markings, intersections, calc, patterns}
\makeatletter
\tikzset{
  use path for main/.code={%
    \tikz@addmode{%
      \expandafter\pgfsyssoftpath@setcurrentpath\csname tikz@intersect@path@name@#1\endcsname
    }%
  },
 use path for actions/.code={%
   \expandafter\def\expandafter\tikz@preactions\expandafter{\tikz@preactions\expandafter\let\expandafter\tikz@actions@path\csname tikz@intersect@path@name@#1\endcsname}%
  },
 use path/.style={%
   use path for main=#1,
   use path for actions=#1,
 }
}
\makeatother
\makeatletter
\tikzset{
    mark position/.style args={#1(#2)}{
        postaction={
            decorate,
            decoration={
                markings,
                mark=at position #1 with \coordinate (#2);
            }
        }
    }
}
\makeatother

\newcommand{\function}[5]{\begin{array}{cccl}          
#1: & #2 & \longrightarrow & #3 \\
    & #4 & \longmapsto & #5 \end{array}}

\newcommand{\Acal}{{\mathcal{A}}}

\newcommand{\Ecal}{{\mathcal{E}}}

\newcommand{\Hcal}{{\mathcal{H}}}

\newcommand{\Pcal}{{\mathcal{P}}}

\newcommand{\Vcal}{{\mathcal{V}}}

\newcommand{\sep}{\mathrm{sep}}

\newcommand{\RR}{\mathbb{R}}

\newcommand{\ZZ}{\mathbb{Z}}

\newcommand{\N}{\mathbb{N}}
\newcommand{\Z}{\mathbb{Z}}
\newcommand{\R}{\mathbb{R}}

\newcommand{\iso}{\cong}

\newcommand{\id}{\textbf{\textit{I}}}
\newcommand{\im}{\operatorname{im}}

\newcommand{\Vol}{\operatorname{Vol}}

\newcommand{\Crit}{\operatorname{Crit}}

\renewcommand{\id}{\operatorname{id}}

\newcommand{\haat}{\widehat}
\newcommand{\tiilde}{\widetilde}

\newcommand{\RBM}{\mathrm{RBM}}
\newcommand{\area}{\mathrm{area}}
\newcommand{\dist}{\mathrm{dist}}
\newcommand{\topo}{\mathrm{top}}

\newtheorem{theorem}{Theorem}
\newtheorem*{theorem*}{Theorem}
\newtheorem{proposition}[theorem]{Proposition}
\newtheorem{lemma}[theorem]{Lemma}
\newtheorem{conjecture}[theorem]{Conjecture}

\newtheorem{corollary}[theorem]{Corollary}

\theoremstyle{definition}
\newtheorem{definition}[theorem]{Definition}
\newtheorem{example}[theorem]{Example}

\newtheorem{assumption}[theorem]{Assumption}
\newtheorem{question}[theorem]{Question}

\theoremstyle{remark}
\newtheorem{remark}[theorem]{Remark}
\newtheorem{claim}[theorem]{Claim}
\numberwithin{equation}{section}

\author{Marcelo R.R. Alves}
\thanks{M.R.R. Alves was supported by the ERC consolidator grant 646649  ``SymplecticEinstein'' and by the Senior Postdoctoral fellowship of the Research Foundation - Flanders (FWO) in fundamental research 1286921N}
\address{Marcelo R.R. Alves, Faculty of Science,\\
University of Antwerp,
 Campus Middelheim,
 Middelheimlaan 1,
BE-2020 Antwerpen,
Belgium.}
\email{\texttt{marcelorralves@gmail.com}}

\author{Lucas Dahinden}
\address{Lucas Dahinden, Mathematisches Institut, Ruprecht-Karls-Universit\"at Heidelberg, Im Neuenheimer Feld 205, DE-69120 Heidelberg}
\email{l.dahinden@gmail.com}
\thanks{L. Dahinden was supported by Deutsche Forschungsgemeinschaft (DFG) under Germany's Excellence Strategy EXC-2181/1-390900948 (the Heidelberg STRUCTURES Excellence Cluster).}

\author{Matthias Meiwes}
\thanks{M. Meiwes was supported by RWTH Aachen University and the  Chair for Geometry and Analysis of the RWTH Aachen University.}
\address{Matthias Meiwes,
Chair for Geometry and Analysis, RWTH Aachen University, Jakobstrasse 2,
DE-52064 Aachen, Germany.}
\email{\texttt{meiwes@mathga.rwth-aachen.de}}

\author{Louis Merlin}
\thanks{L. Merlin was supported by RWTH Aachen University and the  Chair for Geometry and Analysis of the RWTH Aachen University.}
\address{Louis Merlin,
Chair for Geometry and Analysis, RWTH Aachen University, Jakobstrasse 2,
DE-52064 Aachen, Germany.}
\email{\texttt{louis.merlin@hotmail.fr}}

\begin{document}

\title[~]{$C^0$-Robustness of topological entropy for geodesic flows}



\date{\today}

\maketitle

\begin{abstract}
In this paper, we study the regularity of topological entropy, as a function on the space of Riemannian metrics endowed with the $C^0$ topology. We establish several instances of entropy robustness (persistence of entropy non-vanishing after small $C^0$ perturbations).

A large part of this paper is dedicated to metrics on the 2-dimensional torus, for which our main results are that metrics with a contractible closed  geodesic have robust entropy (thus generalizing and quantifying a result of Denvir-Mackay) and that metrics with robust positive entropy on the torus are $C^{\infty}$ generic. Moreover, we quantify the asymptotic behavior of volume entropy in the Teichmüller space of hyperbolic metrics on a punctured torus, which bounds from below the topological entropy for these metrics.

For general closed manifolds of dimension at least 2 we prove that the set of metrics with robust and high positive entropy is $C^0$-large in the sense that it is dense, contains cones and arbitrarily large balls.
\end{abstract}

\tableofcontents

\section{Introduction}

\subsection{Context}

In this paper we study robustness properties for the topological entropy of Riemannian geodesic flows with respect to the $C^0$-topology on the space of Riemannian metrics. 

\emph{The space of metrics} Let $Q$ be a closed manifold and $\mathfrak{G}(Q)$ be the space of $C^\infty$-smooth Riemannian metrics on $S$. For $g,g' \in \mathfrak{G}(Q)$, we say that $g\prec g'$ if $g_x(v,v) \leq g_x'(v,v)$ for all $x,v$, and for $C\in\RR$ we define $Cg \in \mathfrak{G}(Q)$ by $Cg_x(v,w) = C\cdot g_x(v,w)$, for all $x\in Q, v,w \in T_xQ$. We consider on $\mathfrak{G}(Q)$ the metric $d_{C^0}$ defined by 
\begin{equation}
d_{C^0}(g,g') = \inf \left\{\log(C) \, \,  \middle| \,\,    \frac{1}{C}g \prec g' \prec Cg ; \,C>0\right\}.
\end{equation}
The metric $d_{C^0}$ defines the $C^0$ topology on $\mathfrak{G}(Q)$. $d_{C^0}$ is a variant of the Riemannian Banach-Mazur distance from \cite{JunVukasin}, see below for further discussion.  
From a purely metric point of view, $d_{C^0}$ is natural since geometric quantities such as the volume of subsets of $(Q,g)$, the diameter of $(Q,g)$ and the Riemannian distance function $d_g$ on $Q$, are all continuous with respect to $d_{C^0}$.\footnote{Moreover, the logarithms of these quantities are Lipschitz with respect to $d_{C^0}$.}   In studying these quantities it is therefore more natural to consider the $d_{C^0}$-distance than stronger analogues for $C^k$-metrics, $k\geq 1$. 

\emph{Topological entropy} The topological entropy $h_{\topo}$ is a numeric invariant of a dynamical system that measures orbit complexity, see Subsection ~\ref{subsec:Setup} for its definition.  In this paper we study the topological entropy $h_{\rm top}$ of Riemannian geodesic flows $\varphi_g$, $g\in \mathfrak{G}(Q)$, seen as a function on the metric space $(\mathfrak{G}(Q),d_{C^0})$. We will show that, especially when $Q$ is two-dimensional, $h_{\rm top}$ is more robust under perturbation of the metric than one could imagine at first sight.  The reason why this robustness is not clear is that the geodesic vector field depends on the first derivatives of $g$, and therefore does not change continuously with $d_{C^0}$: a $d_{C^0}$-small change of the metric can result in a large change of the geodesic flow, meaning that a priori it is reasonable to believe that a dynamical quantity such as the topological entropy would be subject to a large change. This view is reinforced by the lack of continuity in topologies much stronger than $C^0$: in~\cite{N89} it is shown that on the class of $C^r$ maps, $r<\infty$, $h_{top}$ fails to be upper semi-continuous in $C^\infty$ topology and even smooth perturbations of smooth diffeomorphisms on closed 3-manifolds can collapse topological entropy to~0, see~\cite[Section 2]{Dahinden}, see also~\cite{Milnor}. 

Our investigations are part of the more general study of how the topological entropy of the geodesic flow behaves with respect to perturbations of the metric. This is a long standing problem, and greatly depends on the topology considered on the space of metrics: see for example \cite{katokknieperweiss} and \cite{contreras}. Nowadays, a satisfactory answer is given for metrics of negative curvature and $C^1$ perturbations, even for some non-compact manifolds \cite{schapiratapie}.

\emph{Two continuous invariants}
The first motivation to study the continuity properties of $h_{\rm top}$ on  $(\mathfrak{G}(Q),d_{C^0})$ is that there are two functions on $(\mathfrak{G}(Q),d_{C^0})$  which bound $h_{\rm top}$ from below and which are clearly continuous in $(\mathfrak{G}(Q),d_{C^0})$: the volume entropy $h_{\rm vol}$ and the exponential growth rate $\Gamma_{\rm Morse}$ of the Morse homology of the based loop or free loop space. The fact that $h_{\rm vol}$ is a lower bound for $h_{\rm top}$ is due to Manning \cite{manning}, and the fact that $\Gamma_{\rm Morse}$ is a lower bound for $h_{\rm top}$ is due to Paternain \cite{Paternain} (in the case of based loop spaces) and \cite{Meiwesthesis} (in the case of free loop spaces).
 These two functions either vanish on all of  $\mathfrak{G}(Q)$ or are positive for every element of $\mathfrak{G}(Q)$.  If one of these functions does not vanish, then the topological entropy is \textit{robust} for all $g\in \mathfrak{G}(Q)$ in the sense that for any $g\in \mathfrak{G}(Q)$ there is $c>0$ and an open neighborhood $\mathcal{U}_g$ of $g$ in $(\mathfrak{G}(Q),d_{C^0})$ such  that 
 \begin{equation}\label{eq:robust}
     h_{\topo}(\varphi_{g'})>c \text{ for all } g' \in \mathcal{U}_g.
      \end{equation}
  This shows that for manifolds with positive  $h_{\rm vol}$ or positive $\Gamma_{\rm Morse}$ the topological entropy of geodesic flows cannot be destroyed by $C^0$-small perturbations. Our main results show that some of this robustness persists for manifolds  with vanishing  $h_{\rm vol}$ and $\Gamma_{\rm Morse}$. See for example Theorem \ref{thm:density}, which shows that $C^\infty$ generic metrics on the two-dimensional torus have robust $h_{\rm top}$.

\emph{Homogeneity} 
We proceed to discuss the differences between  $h_{\rm top}$ and the functions $h_{\rm vol}$ and $\Gamma_{\rm Morse}$. 
Recall that the volume entropy $h_{\rm vol}(g)$ of $(Q,g)$ measures the exponential growth rate of the volume of Riemannian balls with respect to the radius on the universal cover of $(Q,g)$. To prove that  $h_{\rm vol}(g)$ is continuous on $(\mathfrak{G}(Q),d_{C^0})$ we first observe that $h_{\rm vol}$ is monotonous: if $g,g' \in \mathfrak{G}(Q)$ and $g\leq g'$ then $h_{\rm vol}(g) \geq h_{\rm vol}(g')$. Furthermore, $h_{\rm vol}$ is homogeneous: $h_{\rm vol}(Cg)=C^{-\frac12}h_{\rm vol}(g)$. The continuity of $h_{\rm vol} $ on $(\mathfrak{G}(Q),d_{C^0})$ follows easily from these two properties. The function $\Gamma_{\rm Morse}$ is also monotonous and homogeneous on $\mathfrak{G}(Q)$, and this was explored in \cite{Dahinden} to study $C^0$-robustness of $h_{\rm top}$ of geodesic flows, and more generally Reeb flows. Since monotonous and homogeneous functions on $\mathfrak{G}(Q)$ are either $0$ or always positive, the homogeneous function $h_{\rm top}$ cannot be monotonous for metrics on the torus. 
Moreover, Theorem~\ref{thm:EntropyDenseAndBig} shows that it is possible to increase $h_{\rm top}$ arbitrarily with $C^0$-small perturbations on $(\mathfrak{G}(Q),d_{C^0})$.
In particular this implies that $C^0$-near any metric we can find another metric that doubles its entropy. It follows that
\begin{corollary}[From Theorem~\ref{thm:EntropyDenseAndBig}]\label{cor:NotHomogeneous}
    Topological entropy of geodesic flows is not homogeneous on any closed manifold of dimension at least 2.
\end{corollary}

\emph{The set of high entropy metrics}
The results of this paper show that on the other hand in many situations $h_{\rm top}$ can not be arbitrarily decreased by small perturbations on $(\mathfrak{G}(Q),d_{C^0})$. Theorem~\ref{thm:generic} shows that on the 2-torus, a generic metric has robust positive topological entropy and Theorem~\ref{thm:EntropyDenseAndBig} shows that on any manifold of dimension at least 2 the set of robust high entropy metrics is $C^0$-dense.

Our results and this discussion suggest the following conjecture:

\begin{conjecture}\label{conj1}
If $Q$ is a closed surface, then $h_{\rm top}(g)$ is robust whenever it does not vanish.
\end{conjecture}

Although all methods presented in this paper yield robust lower bounds which are not sharp, we ask also the following question, that, if answered in the affirmative, implies Conjecture~\ref{conj1}.    
\begin{question}
Is, for any closed surface $Q$,  $h_{\rm top}$ lower semi-continuous on $(\mathfrak{G}(Q),d_{C^0})$?
\end{question}

The $C^0$ distance $d_{C^0}$ is a variant of the Riemannian Banach-Mazur distance $d_{\rm RBM}$ defined by Stojisavljevi\'c and Zhang in \cite{JunVukasin}, a pseudo-metric on $\mathfrak{G}(Q)$.\footnote{We note that the continuity and local robustness results for $h_{\topo}$ in this paper still hold when considering $d_{\rm RBM}$ instead of  $d_{C^0}$.} $d_{\rm RBM}$ is defined as $d_{\rm  RBM}(g,g') := \inf d_{C^0}(g,\varphi^*g')$, where the infimum is taken over all diffeomorphisms $\varphi:Q\to Q$. The pseudo-metric $d_{\rm RBM}$ itself is an adaption to $\mathfrak{G}(Q)$ of the symplectic Banach-Mazur distance, that was first proposed by Ostrover and Polterovich to study the symplectic geometry of Liouville domains and studied e.g. in \cite{JunVukasin, Usher}.
In \cite{JunVukasin} the authors investigate the large scale geometry of $(\mathfrak{G}(Q), d_{\rm RBM})$ and one of their result is that for $Q=T^2$ and every  $n\in \N$ there is a quasi-isometric embedding $\Phi_n:(\R^n, |\cdot|_{\infty}) \to (\mathfrak{G}(T^2), d_{\rm RBM})$\footnote{In fact, it is an  embedding into the space of metrics with fixed volume and diameter bounded by a fixed constant.}, so informally speaking $(\mathfrak{G}(T^2),d_{\rm RBM})$ is "very large" in the metric sense. The construction in \cite{JunVukasin} can be easily modified to have its image in the set of high entropy metrics by adding a $C^0$-small modification, cf.\ Corollary~\ref{cor:EntropyDenseAndBig}.

\emph{Ma\~n\'e's formula for the topological entropy }
In \cite{Mane}, Ma\~n\'e established the following remarkable formula for the topological entropy of the geodesic flow of a $C^\infty$-smooth Riemannian metric $g$ on a manifold $Q$:
\begin{equation} \label{eq:Mane}
h_{\rm top}(\phi_g) = \lim_{T \to +\infty } \frac{1}{T} \int_{Q \times Q} \log(\mathcal{N}^g_T(p,q)) d\omega_g(p) d\omega_g(q).  
\end{equation}
Here, $\mathcal{N}^g_T(x,y)$ denotes the number of geodesic chords of $g$ from $p$ to $q$ with length $<L$, $d\omega_g(p)$ means integration in the variable $p$ with respect to the Riemannian volume form $\omega_g$ on $Q$ associated to $g$, and $d\omega_g(q)$ means integration in the variable $q$ with respect $\omega_g$. This formula gives a characterisation $h_{\rm top}(\phi_g)$ in terms of the purely geometric quantity which appears in the right side of \eqref{eq:Mane}.

The right side of \eqref{eq:Mane} is the exponential growth of the average number of geodesics connecting two points of $Q$. Using Ma\~n\'e's formula our results provide a surprising robustness of this exponential growth in case $Q$ is a surface. For example, Theorem \ref{thm:generic} imply that for a $C^\infty$-generic metric $g$ on $T^2$ this exponential growth is positive and cannot be completely destroyed by $C^0$-small perturbation of $g$. This is far from obvious, given that the counting function $\mathcal{N}^g_T(p,q)$ can undergo dramatic changes when we make a $C^0$-small perturbation of the metric.

\subsection{Results, strategy and layout of the paper.}

To highlight a metric $g$ with robust entropy, we proceed in two steps. First we show a forcing type argument, which is a geometric feature of $g$ that implies positivity of the entropy. One of best known example is a theorem of \cite{denvirmackay} stating that a metric on the torus which admits a contractible closed (non constant) geodesic must have positive entropy. Then, we show that the forcing situation is $\mathcal{C}^0$ robust, that is, persists after $d_{\rm C^0}$-small perturbation of $g$.

The analogy with \cite{denvirmackay} is not incidental. In Section~\ref{sec:DenvirMacKay} we describe how a closed contractible geodesic in the 2-torus forces robust topological entropy:

\begin{definition}
Let $\Pi: \R^2 \to T^2 =  \R^2 / \Z^2$ be the canonical projection and $g_{\rm flat}$ the pushforward by $\Pi$ of the euclidean metric. 
For any metric $g\in \mathfrak{G}(T^2)$  define $D(g) = e^{d_{C^0}(g_{\rm flat}, g)}$. 
\end{definition}
    
\begin{theorem} \label{thm:robustDM}

Let $g_0$ be a Riemannian metric on $T^2$ with a closed contractible geodesic. Then $g_0$ has robust topological entropy. 

Moreover the following holds. Given $\delta>0$ there is $\epsilon>0$ such that for all $g$ with $d_{C^0}(g,g_0) < \epsilon$
\begin{enumerate}
  \item $h_{\topo}(g) > ({\lceil \frac{\sqrt{D(g_0)}\Lambda}{2} + \delta \rceil } {\sqrt{D(g_0) + \delta}})^{-1}\log(3)$, 
  where $\Lambda = l_{g_0}(\gamma_0)$ if $g_0$ is bumpy and $\Lambda= 2(l_{g_0}(\gamma_0)+ \sqrt{D(g_0)})$ in general.
  \item $h_{\topo}(g) > \min\left\{\frac{1}{\sqrt{4\area_{g_0}(T^2) + L^2}},\frac{2}{3L} \right\} \log(2)$, where $L= l_{g_0}(\gamma_0)$ if $g_0$ is bumpy and $L= \max\{4\sqrt{4\area_{g_0}(T^2) + l_{g_0}(\gamma_0)^2}, 3l_{g_0}(\gamma_0)\}$ in general.  
  \end{enumerate}
\end{theorem}

The two parts of this theorem have similar, yet different core ideas. For a clean exposition, we prove the first part of this theorem first in the case where $g_0$ is bumpy  in Subsection~\ref{subsec:nondeg}, then in Subsection~\ref{subsec:deg} in the degenerate case. We then prove the second part of the theorem in Subsection~\ref{subsec:areabound}. 

A corollary of Theorem \ref{thm:robustDM} which we believe to be interesting in its own right is the following:
\begin{corollary} \label{cor:entropyrigidity}
Let $g$ be a Riemannian metric on the $2$-torus whose area is $\leq 1$ and which has a closed contractible geodesic whose length is  $\leq 1$. Then, $h_{\rm top}(\phi_g) \geq \frac{1}{20}$.
\end{corollary}
An interesting question is to know if this result remains true without any assumption on the length of the closed contractible geodesic, i.e., if there is a lower bound for the $h_{\rm top}$ of geodesic flows of Riemannian metrics on the $2$-torus with area $1$ and which have a closed contractible geodesic. 

One result which is needed for Theorem \ref{thm:robustDM} and throughout the paper is presented in Appendix~\ref{appendix:Spectrum}, where we prove $C^0$-robustness of the length spectrum of a Riemannian manifold with bumpy metric. We do this with the aim to search the lowest possible amount of technology with which to show robustness of the length spectrum in the non-degenerate case. The main statement is the following, for the definition of topological non-degeneracy see Definition~\ref{def:topologicalNonDegeneracy}.
\begin{proposition}\label{stabl}
Suppose that $0\neq e\in (a,b)$ is the only energy value of a closed geodesic, and that all geodesics with energy $e$ are (topologically) non-degenerate. Then, any $C^0$-close Riemannian metric has a closed geodesic in the same homotopy class with energy close to $e$.
\end{proposition}
\begin{remark}
    For the first bounds in Theorem~\ref{thm:robustDM} the condition for $g_0$ to be bumpy is stronger than necessary; It would suffice to ask that only $\gamma_0$ is non-degenerate. Correspondingly, in Proposition~\ref{stabl} one may ask that $\gamma_0$ is topologically non-degenerate and isolated in the loop space instead of its energy value to be isolated in the energy spectrum. However, this would significantly complicate the proof of the proposition, cf.~Remark~\ref{rem:localization}, and since we also treat the degenerate case, this additional complication would yield no significant improvement for our theorem.
\end{remark}


In Section~\ref{sec:hyperbolic}, we turn our attention to the {\it volume entropy} (to be defined in paragraph \ref{subsec:Setup}) for hyperbolic metrics on the one-holed torus. The volume entropy is a lower bound for the topological entropy, so that a metric with robust $h_{\rm vol}$ has in particular robust $h_{\rm top}$. Consider again the set of metrics on the 2-torus such that they admit a non-constant simple closed contractible geodesic. This geodesic is in particular separating, one component of its complement is a disk, the other is a one-holed torus. In Section~\ref{sec:hyperbolic}, we assume that the metric on the one-holed torus is hyperbolic (constant curvature -1). Denote by $\mathcal{H}$ this set of metrics on the torus. We prove the following result.

\begin{theorem}\label{thm:volumeentropynoncollapsing}
Any Riemannian metric $g\in\mathcal{H}$ has robust $h_{\rm vol}$ and thus robust $h_{\rm top}$.
\end{theorem}

Following the general approach (forcing plus robustness), in Section~\ref{sec:propertyF}, we show that a certain configuration of curves in the torus, which we call a ribbon, forces robust topological entropy. Moreover, this condition is $C^\infty$ generic, leading to a series of theorems that are more precisely stated in Theorem~\ref{thm1}-\ref{thm:generic}.

\begin{theorem} \label{thm:density}
Four closed geodesics on the 2-torus that form a ribbon force robust~$h_{\rm top}$. 
A $C^\infty$ generic metric possesses four closed geodesics that form a ribbon and has thus robust $h_{\rm top}$.
\end{theorem}

In Section~\ref{sec:retractibleNeck}, we find a robust (albeit non-generic) forcing condition, which we call retractable neck with entropic body on a Riemannian manifold of any dimension. The following theorem is more precisely stated in Theorem~\ref{thm:entropyFromNeck}:
\begin{theorem}
	Let the closed Riemannian manifold $(M,g)$ (of any dimension) have a retractable neck and entropic body. Then, $g$ has robust $h_{\rm top}$.
\end{theorem}

This condition readily comes with an explicit construction, see Example~\ref{ex:addHead}, which allows to prove the following statements on the size of the space of metrics with high entropy which we denote by $\mathfrak{G}^e(Q)=\{g\in\mathfrak{G}(Q)\mid h_{\rm top}(g)>e\}$

\begin{theorem}\label{thm:EntropyDenseAndBig}
    Given any closed manifold $Q$ of dimension at least 2. For any $e>0$ and for any metric $g\in\mathfrak{G}(Q)$, there is a $C^0$-continuous path $g(s):(0,\infty)\to\mathfrak{G}^e(Q)$ with $d_{C^0}(g(s),g)<s$ such that for all $s$ the $d_{C^0}$-ball of radius $r(s)=\frac{s+3}{2+(s+1)e^{-s/8}}$ around $g(s)$ has high entropy: 
    \[\bigcup_s B_{r(s)}(g(s))\subseteq \mathfrak{G}^e.\]
\end{theorem}
Please note that in our construction $r(s)$ does not depend on $e$, but the path $g(s)$ does.

\begin{corollary}\label{cor:EntropyDenseAndBig}
    \begin{itemize}
        \item For any $e>0$, $\mathfrak{G}^e(Q)=\{g\in\mathfrak{G}(Q)\mid h_{\rm top}(g)>e\}$ is $C^0$-dense in $\mathfrak{G}(Q)$. 
        \item As $\lim_{s\to\infty}r(s)=\infty$, we find arbitrarily big balls of arbitrarily high entropy.
        \item For $Q=T^2$ let $\overline{\mathfrak{G}}(T^2)$ be the set of metrics with volume 1 and diameter $\leq 101$. Then, for every $n\in \N$ there is a quasi-isometric embedding $\Phi_n:(\R^n, |\cdot|_{\infty}) \to (\overline{\mathfrak{G}}^e(T^2), d_{\rm RBM})$.
    \end{itemize}
\end{corollary}

\subsection{Related developments}

In an ongoing joint project of the authors Alves, Dahinden and Meiwes with Abror Pirnapasov, we are working to generalise some of the results of the present paper, such as item $(1)$ of Theorem \ref{thm:robustDM} and Theorem \ref{thm:generic}, to the category of Reeb flows. Reeb flows on contact $3$-manifolds are a generalisation of geodesic flows of Riemannian metrics on surfaces, and the $C^0$-distance on the space of contact forms considered in \cite{Dahinden} generalises to the space of Reeb flows on unit tangent bundles of surfaces (endowed with the geodesic contact structure) the $C^0$-distance on the space of Riemannian metrics of surfaces that we consider here. For this generalisation one must use symplectic topological methods developed in \cite{AP} to study $h_{\rm top}$ of Reeb flows. 

On the other hand Corollary \ref{cor:entropyrigidity} cannot be generalised to the category of Reeb flows. Using the methods of \cite{AASS} one can construct Reeb flows on the $3$-torus $(T^3,\xi_{\rm geo})$ endowed with the geodesic contact structure which contradict any reasonable generalisation of  Corollary \ref{cor:entropyrigidity}.

The questions considered in the present paper were inspired by \cite{Dahinden} and \cite{ChorMeiwes}.

\subsection{Setup and Definitions}\label{subsec:Setup}

Let $(Q,g)$ be a compact Riemannian manifold.
Throughout this paper, we will be interested in ergodic properties of its geodesic flow. Let $T^1Q$ denote the unit tangent bundle of $Q$. For a vector $v\in T^1Q$, we consider the geodesic $\gamma_v$ defined by the initial condition $\gamma_v'(0)=v$.

The geodesic flow of $(Q,g)$, denoted $\varphi^t_g$ (we may sometimes omit $t$ or $g$ when the context is clear) is defined by
\[\function{\varphi^t_g}{T^1Q}{T^1Q}{v}{\gamma_v(t)}.\]

When $Q$ is a manifold with boundary, we restrict $\varphi^t$ to the recurrent set in $T^1Q$.

\subsubsection*{Entropies}
Denote by $\Gamma_t f(t):= \limsup_t \frac 1t \log f(t)$ the exponential growth in $t$ of a function $f(t)$. We use the following definition of topological entropy:
\begin{definition}\label{def:topologicalEntropy}
	Let $\varphi:(M,d)\to (M,d)$ be a continuous self map of a compact metric space. Define the dynamical metric 
	\[d_k(x,y)=\sup_{0\leq l\leq k} d(\varphi^lx,\varphi^ly).\]
	A $(\delta,k)$-separated set is a set whose points have pairwise $d_k$-distance $\geq \delta$. Topological entropy of $\varphi$ is then defined as follows:
	\[h_{top}(\varphi)=\lim_{\delta\to 0} \Gamma_k \sup_{\Delta} \abs{\Delta},\]
	where the supremum runs over all $(\delta,k)$-separated sets.
There exists a "dual" characterization of the entropy with spanning sets instead of separated sets. Note that for any fixed $\delta>0$ any sequence $\Delta_k$ of $(\delta,k)$-separated sets provides the lower bound 
	\[h_{\rm top}(\varphi)\geq\Gamma_k|\Delta_k|\] 
	on the entropy. This is how we obtain all lower bounds to the entropy in this paper. The difficulty is to find a growing sequence, i.e., one that provides a positive lower bound.

	Topological entropy is well known to be independent of the choice of metric (generating the same topology). Further, if $\varphi:\RR\times M\to M$ is an autonomous flow, then it is well known that 
	\[h_{top}(\varphi^1)=\frac 1T h_{top}(\varphi^T)\]
	for any $T>0$.
	
	We say the topological entropy $h_{top}(g)$ of a geodesic flow on a Riemannian manifold is the topological entropy of its time-1 geodesic spray in the unit sphere bundle.
	\flushright$\triangle$
\end{definition}

Volume entropy is a related invariant, more geometric in nature. Let $M$ be a closed manifold (possibly with non empty boundary) and $\widetilde{M}$ its universal cover that we endow with the metric pulled-back from $M$. We choose a basepoint $x\in\widetilde{M}$.

\begin{definition}\label{def:hvol}
By a result of \cite{manning}, the following limit exists, is independent of the basepoint $x$ and we call it the volume entropy of $M$.
\[h_{\rm vol}=\lim_{R\to \infty}\frac{\log \Vol B(x,R)}{R}\]
where $B(x,R)$ is the ball of radius $R$ centered at $x$ in the universal cover $\widetilde{M}$.

\flushright $\triangle$
\end{definition}

Comparison between $h_{\rm vol}$ and $h_{\rm top}$ is done in \cite{manning}: we have
\[h_{\rm vol}\leq h_{\rm top}\]
for any closed manifold $M$. Note that the result also applies to manifolds with boundary. Even if this case is not stated in \cite{manning}, the same proof extends to this case. Equality happens, e.g., when $M$ has nonpositive curvature.

In the context that we have in mind, the above inequality simply reformulates as the fact that a metric with robust $h_{\rm vol}$ also has robust $h_{\rm top}$.

\subsubsection*{Length, energy, area and loop spaces}
Let $Q$ be a closed manifold. Let $g$ be Riemannian metric on $Q$
The \textit{length} resp. \textit{energy} of a smooth curve $x:[0,1] \to Q$ with respect to $g$ are
\begin{align*}
	l_g(x) &= \int_0^1 \sqrt{ g(x'(t), x'(t))} \, dt, \text{ resp. } \\ 
		\mathcal{E}_g(x) &= \frac{1}{2}\int_{0}^1 g(x'(t),x'(t)) \, dt.
\end{align*}
Note that, for $C>0$, $l_g(x) = \frac 1{\sqrt C}l_{Cg}(x)$, and $\mathcal{E}_g(x) =\frac 1C\mathcal{E}_{Cg}(x)$, where $Cg$ is the metric $Cg(v,w) := C(g(v,w))$. 

If $Q$ is a surface and $\Sigma\subset Q$ a subsurface, denote by  $\area_g(\Sigma) = \int_{\Sigma} \sigma_g$ the \textit{area} of $\Sigma$, where $\sigma_g$ is the Riemannian volume form of $g$.   

We use the following notations for various versions of loop spaces:
\begin{align*}
	\mathcal{L}Q &:= \{ \gamma:S^1 \to Q \, | \, \gamma \text{ smooth}\},\\
	\mathcal{L}Q^{<(\leq)a}_g &:= \{\gamma \in  \mathcal{L}Q\, |\, \mathcal{E}_g(\gamma) <(\leq) a\},\\
	\mathcal{L}_{\alpha}Q &:= \{ \gamma \in \mathcal{L}Q \, | \, [\gamma] = \alpha\},\\
	\mathcal{L}_{\alpha}Q^{<(\leq)a}_{g} &:= \{\gamma \in  \mathcal{L}_{\alpha}Q \, |\, \mathcal{E}_g(\gamma) <(\leq) a\}.	
\end{align*}

\subsubsection*{Robustness}
We are interested in conditions under which the entropy of the geodesic flow is robust under $C^0$-perturbation of the metric. Let us formulate the robustness property in which we are interested. Let $Q$ be closed manifold, equipped with a metric $g$. 

\begin{definition}\label{def:robust}
Let $\varepsilon>0$. We say that $g$ has $\varepsilon$-robust $h_{\rm top}$ if there is $c>0$ such that for all metrics $g'$ with $d_{\RBM}(g,g') < \varepsilon$ we have $h_{\rm top}(g')\geq c$. For abbreviation we say that $g$ has robust $h_{\rm top}$ if it has $\varepsilon$-robust $h_{\rm top}$ for some $\varepsilon>0$.
\end{definition}
\begin{remark}
	It seems that there is no established terminology for this property in the literature. In other places, $g$ may be said to have \emph{stable} $h_{\rm top}$ or to be \emph{entropy non-collapsing}.
\end{remark}

\subsection*{A preliminary robustness lemma}
We describe the (classic) mechanism that deduces positive topological entropy from many homotopically different periodic orbits. 

Two free homotopy classes $\alpha,\beta$ are coprime if they are not multiple covers of a common class~$\gamma$. Equivalently, they do not possess homotopic multiple covers $n\alpha\neq m\beta$.

\begin{lemma}\label{lem:nonhomotopicGeodesicsToEntropy}
	Let $S_g$ be a compact Riemannian manifold. Let $\Pcal_g$ be a set of pairwise periodic $g-$geodesics in pairwise coprime homotopy classes and let $\{\Pcal^L_g\}_{L\in\RR}$ be the filtration by $g$-length. 

	Suppose that $\Gamma_L (\#{\Pcal_g^L})\geq\gamma$. Then, $h_{top}(g)\geq \gamma$.
\end{lemma}

\begin{proof}
	This follows from \cite[Theorem 1]{Alves1}. It can also be obtained using the same argument used to prove Manning's inequality in \cite{manning}.
\end{proof}


\section{A robust version of Denvir-MacKay theorem}\label{sec:DenvirMacKay}

The aim of this section is to show Theorem~\ref{thm:robustDM}. 
This builds upon~\cite{denvirmackay}, where it is shown that a closed contractible simple geodesic implies the existence of many other geodesics. We first prove in Subsection~\ref{subsec:nondeg} item (1) of Theorem \ref{thm:robustDM} in case $\gamma_0$ is topologically non-degenerate.
For the notion of topological non-degeneracy see Definition~\ref{def:topologicalNonDegeneracy}.
In Subsection~\ref{subsec:areabound}, we provide an explicit robustness constant in terms of area and minimal length of closed contractible geodesic that is stated item (2) Theorem \ref{thm:robustDM}.

\subsection{Proof of item (1) of Theorem \ref{thm:robustDM} in case $\gamma_0$ is non-degenerate}\label{subsec:nondeg}
The strategy is to combine the robustness of closed contractible geodesics with the proof that a closed contractible orbit implies positive topological entropy. 

Before presenting the proof, we recall some preliminary notions which will be needed in the proof. We fix, once and for all, a covering map $\Pi: \mathbb{R}^2 \to T^2 $, such that the group $G$ of deck transformations associated to $\Pi$ is the group composed by translations of $\mathbb{R}^2$ by vectors $(m,n)$ where $m$ and $n$ are integers. In other words 
$$G= \{\mathcal{T}_{(m,n)} \ | \ m \mbox{ and } n \in \mathbb{Z}\},$$
where $\mathcal{T}_{(m,n)}: \mathbb{R}^2 \to \mathbb{R}^2$ is given by $\mathcal{T}_{(m,n)}(x,y) =(x+m,y+n)$.
It is clear that for any choice of real numbers $a$ and $b$ the unit square $[a,a+1] \times [b,b+1] \subset \R^2$ is a fundamental domain for the covering $\Pi$.

 Let $\sigma$ be an immersed closed contractible curve in $T^2$ with only transverse self-intersections. A {\it lift} of $\sigma$ is an immersed closed curve $\widetilde{\sigma}$ in $\R^2$, such that for any parametrisation ${\mathfrak{f}}:S^1 \to \widetilde{\sigma}$ of $\widetilde{\sigma}$ the composition $\Pi \circ \mathfrak{f}$ is a parametrisation of $\sigma$.
 
We will need the following elementary lemma. 
\begin{lemma} \label{lemma:universal}
Let $g$ be a Riemannian metric on $T^2$ and fix a number $l >0$. Then,
if $\sigma$ is an immersed closed curve in $T^2$ with only transverse self-intersections and $g$-length $<l$, every lift $\widetilde{\sigma}$ of $\sigma$ is contained in a square of the form $[a+ \lceil  \frac{D(g)l}{2} \rceil, b +  \lceil  \frac{D(g)l}{2} \rceil ]$.
\end{lemma}

Recall that we defined $D(g)=e^{d_{C^0}(g_{\rm flat},g)}$. 

\proof
Since the $g$-length of $\sigma$ is $<l$, the $g_{\rm flat}$-length of $\sigma$ is $<D(g)l$, which implies that any lift $\widetilde{\sigma}$ of $\sigma$ has length $<D(g)l$ with respect to the flat metric in $\R^2$. Let $p_{\rm left}$ be a leftmost point of $\widetilde{\sigma}$ and $p_{\rm right}$ be a rightmost point of $\widetilde{\sigma}$. Since $\sigma$ has length $<D(g)l$ and has to travel from $p_{\rm left}$ to $p_{\rm right}$ and back to $p_{\rm left}$ by running a distance smaller then its length, we conclude that the Euclidean distance between $p_{\rm left}$ to $p_{\rm right}$ is $< \frac{D(g)l}{2} $. We conclude that $\widetilde{\sigma}$ is contained in a vertical strip of $\R^2$ with width $< \frac{D(g)l}{2}$. 

Reasoning similarly for an uppermost point and a lowermost point of  $\widetilde{\sigma}$ we conclude that $\widetilde{\sigma}$ must be contained in a horizontal strip of width $<\frac{D(g)l}{2}$. This establishes the lemma. \qed

\

We are now ready to proceed with the proof.

{\bf Step 1: Robustness of the length spectrum.} If $\gamma_0$ is a closed contractible topologically non-degenerate geodesic of $g_0$, then for a sufficiently small $\delta>0$ there is an $\epsilon>0$ such that every Riemannian metric $g$ whose $d_{ C^0}$-distance to $g_0$ is $< \epsilon$ has a closed contractible geodesic $\gamma$ with $|l_g(\gamma)-l_{g_0}(\gamma_0)|<\delta$ and  $|l_{g_0}(\gamma)-l_{g_0}(\gamma_0)|<\delta$. Since $D(g)$ also varies continuously with respect to $d_{C^0}$, we can choose $\epsilon>0$ to be small enough so that $D(g) < D(g_0) + \delta$.

This is well known. In Appendix~\ref{appendix:Spectrum} we present a proof using very low-tech strategies. 

{\bf Step 2: Simplifying the geodesic.}
We start with the closed contractible geodesic $\gamma$ for $g$ from Step 1. Denote by $\tilde\gamma$ a lift to the universal cover $\R^2$, which is a closed geodesic of $\widetilde{g}:=\Pi^* g$ because $\gamma$ is contractible. Denote by $S$ the closure of the unbounded component of $\R^2 \  \setminus \ \tilde\gamma$, a set that is homeomorphic to a plane minus an open disk. Let $\alpha$ be one of the two non-trivial free homotopy classes of loops in $S$ which  contains embedded loops; these are the classes of curves in $S$ which encircle once the disk which was removed from $\R^2$.
Using the curve shortening flow and reasoning as in the proof of Lemma~2 of~\cite{Angenent3} one obtains a contractible simple closed geodesic $\gamma_\alpha$ of $\widetilde{g}$ in the homotopy class $\alpha$, whose $g$-length is the minimal possible length for all curves in the homotopy class $\alpha$.

Combining this discussion with Lemma \ref{lemma:universal}, we conclude that there $\widetilde{\gamma}$ is contained in a square of the form $[a+  \lceil  \frac{\sqrt{D(g)} l_{g(\gamma)}}{2} \rceil, b +  \lceil  \frac{{\sqrt{D(g)}l_{g(\gamma)}}}{2} \rceil ]$. For simplicity, we let $N:=  \lceil\frac{{\sqrt{D(g)}l_{g(\gamma)}}}{2} \rceil$. 

Let $\widehat{G}$ be the subgroup of $G$ generated by $\mathcal{T}_{N,0}$ and $\mathcal{T}_{0,N}$. Since $\gamma_\alpha$ is a simple closed curve in $\R^2$ contained in a fundamental domain $[a+ N , b +  N ]$ of $\widehat{G}$, it is a simple contractible closed geodesic in the quotient Riemannian manifold $(\widehat{T},\widehat{g})$ that is obtained by quotienting $(\R^2,\widetilde{g}:=\Pi^*g)$ by the action of  $\widehat{G}$.



{\bf Step 3: Many free homotopy classes.}
Since $\widetilde{\gamma}$ is a simple contractible closed geodesic in $\widehat{T}$, it bounds a disk $\widehat{D}$ in $\widehat{T}$. The surface ${S}:= \widehat{T} \setminus \widehat{D}$ is diffeomorphic to the torus minus a disk, so that $\pi_1{S}$ is the free group with two generators. Let $\alpha_1$ be the projection of $[a,a+N] \times \{b\}$ to $\widehat{T}$ and $\alpha_2$ be the projection of $\{a\} \times [b,b+N] $ to $\widehat{T}$: $\{\alpha_1,\alpha_2\}$ is a basis of $\pi_1{S}$.

We proceed to estimate the $\widehat{g}$-length of $\alpha_1$ and $\alpha_2$.
It is clear that  \[{l}_{\widehat{g}}(\alpha_1)=N{l}_{\widetilde{g}}([a,a+1] \times \{b\}) \quad \mbox{and}\quad {l}_{\widehat{g}}(\alpha_2)=N{l}_{\widetilde{g}}( \{a\} \times [b,b+1] ).\] 
We know that $ \max\{{l}_{\widetilde{g}}(\{a\} \times [b,b+1]), {l}_{\widetilde{g}}( [a,a+1] \times \{b \}) \} \leq {\sqrt{D(g)}}$.

Recall that the set of free homotopy classes of loops $\Omega(S)$ in $S$ equals the set of conjugacy classes of $\pi_1(S)$. Given a number $n$ we denote by $\Omega_n(S)$ to be the set of elements in $\Omega(S)$ which have at least one representative in $\pi_1(S)$ with word length $\leq n$. It is well-known that $\# \Omega_n(S) \geq \frac{8(3)^{n-2}}{n}$. 

Let $\rho \in \Omega_n(S)$. Because we can represent $\rho$ as a word in $\alpha_1$ and $\alpha_2$ with word length $\leq n$, we can find curves in $\rho$ whose length is $\leq  N{e^{D(g)}}n$.
 Since $(S,\widehat{g}) $ is a Riemannian surface with geodesic boundary, there exists a minimizing closed geodesic $\gamma_\rho$ of  $(S,\widehat{g})$ in $\rho$ contained in the interior of $S$. The length of $\gamma_\rho$ is  $\leq  N{\sqrt{D(g)}}n$

{\bf Step 4: Many geodesics lead to entropy.}
Let $\mathcal{N}^C(\widehat{g})$ be the set of prime minimizing closed geodesics of length $\leq C$ in $(S,\widehat{g})$ and  $\widetilde{\mathcal{N}}^C(\widehat{g})$ be the set of minimizing closed geodesics of length $\leq C$ in $(S,\widehat{g})$. A simple computation shows that the exponential growths $\limsup_{C \to +\infty} \frac{\log (\# \mathcal{N}^C(\widehat{g}))}{C} $ of $\mathcal{N}^C(\widehat{g})$ and  $\limsup_{C \to +\infty} \frac{\log (\# \widetilde{\mathcal{N}}^C(\widehat{g}))}{C} $ of $\widetilde{\mathcal{N}}^C(\widehat{g})$ are the same.

But $$ \limsup_{C \to +\infty} \frac{\log (\# \widetilde{\mathcal{N}}^C(\widehat{g}))}{C} \geq \frac{1}{N\sqrt{D(g)}}\limsup_{n \to +\infty} \frac{\log (\# \Omega_n(S) ) }{n} \geq \frac{\log(3)}{N\sqrt{D(g)}}.$$

 Let $\varphi_{\widehat{g}}$ be the geodesic flow of $\widehat{g}$ on $T^1 \widehat{T}$. We consider the set of globally minimizing geodesics of $(S, \widehat{g})$. The lift of this set is a compact invariant set $\mathcal{C}$ of $\varphi_{\widehat{g}}$ completely contained in $T^1 S$. The lift of every minimizing closed geodesic of $(S,\widehat{g})$ belongs to $\mathcal{C}$.
Reasoning as in the proof of Manning's inequality one obtains that the exponential growth of $\mathcal{N}^C(\widehat{g})$ is a lower bound for $h_{\rm top}(\varphi_{\widehat{g}}|_{\mathcal{C}})$. Noting that two different prime closed geodesics of $(S,\widehat{g})$ give closed trajectories of $\varphi_{\widehat{g}}$  in $\mathcal{C}$ which are not homotopic in $T^1 S$, one can also obtain that the exponential growth of $\mathcal{N}^C(\widehat{g})$ is a lower bound for $h_{\rm top}(\varphi_{\widehat{g}}|_{\mathcal{C}})$ by using Theorem 1 of \cite{Alves1}. We have thus concluded that $$ h_{\rm top}(\varphi_{\widehat{g}}) \geq  h_{\rm top}(\varphi_{\widehat{g}}|_{\mathcal{C}}) \geq  \frac{\log(3)}{N\sqrt{{D(g)}}}. $$ Since $ h_{\rm top}(\varphi_{\widehat{g}}) =  h_{\rm top}(\varphi_{{g}})$ we have shown that for every Riemannian metric $g$ with $d_{C_0}(g,g_0) < \epsilon$ we have $h_{\rm top}(\varphi_{{g}}) \geq \frac{\log(3)}{N{\sqrt{D(g)}}}$, which completes the proof of the theorem. \qed 



\subsection{Proof of item (1) of Theorem \ref{thm:robustDM} in case $\gamma_0$ is degenerate}\label{subsec:deg}

In the degenerate case we face the problem that the closed contractible geodesic may not persist after perturbation of the metric. It is easy to construct examples where it does not. However, in the following proof we are able to show that there are different contractible geodesics that will persist. The proof follows the scheme of the non-degenerate case.

\begin{proof}

{\bf Step 1: Simplifying the geodesic.} 
First, it is sufficient to consider the case where the lift $\widetilde{\gamma_0}$ to $\R^2$ is a simple closed contractible curve. This follows from \cite{Angenent3} who showed any Riemannian metric on $T^2$ with a closed contractible geodesic $\alpha$ with length $l$ has a a closed contractible geodesic $\alpha'$ whose length is $\leq l$ and whose lift to $\R^2$ is simple. For completeness we sketch the proof of this fact. If a lift $\widetilde{\alpha}$ of $\alpha$ is simple there is nothing to be done. Otherwise, letting $\Sigma$ be the closure of the unbounded component of $\R^2 \setminus \widetilde{\alpha} $, we know that $\partial \Sigma$ is a geodesic polygon formed by pieces of $\widetilde{\alpha}$, and the length of $\partial \Sigma$ is clearly smaller than $l$. Perturbing $\partial \Sigma$ to a curve completely contained in the interior of $\Sigma$ and applying the curve shortening flow one obtains the desired $\alpha'$.

Reasoning as in the proof of the non-degenerate case we obtain a square of the form $[a, a+ \lceil \frac{\sqrt{D(g_0)}}{2} \rceil ] \times [b, b+ \lceil \frac{\sqrt{D(g_0)}}{2} \rceil ] $. We let $(\widehat{T}^2,\widehat{g}_0)$ be the quotient of $(\R^2,\pi^*g_0)$ by the group of translations generated $\mathcal{T}_{0,M }$ and $\mathcal{T}_{M \rceil,0 }$, where $M:= \lceil \frac{\sqrt{D(g_0)}}{2} \rceil$.

We denote by $S$ the non-contractible component of $\widehat{T}^2\backslash  {\rm im}(\gamma_0)$ and by $D$ its complement. Note that $\partial D\subseteq {\rm im}(\gamma_0)$.\\\

{\bf Step 2: Finding the geodesic.} 

Fix a lift $\widetilde{\gamma_0}$ and let $\widetilde{\gamma_0}'$ be the the closest lift of ${\gamma_0}$ which does not intersect  $\widetilde{\gamma_0}$. The distance between $\widetilde{\gamma_0}$ and  $\widetilde{\gamma_0}'$ is $\leq \sqrt{D(g)}$: this is an elementary exercise and we leave this as an exercise to the reader. We let $\iota$ be the shortest geodesic of $\pi^*g_0$ connecting $\widetilde{\gamma_0}$ and  $\widetilde{\gamma_0}'$.  Let then $\varsigma$ be a simple closed curve in $\R^2$ contained in a small tubular neighbourhood of $\widetilde{\gamma_0} \cup \widetilde{\gamma_0}'\cup \iota$, and such that the bounded component of $\R^2 \setminus \varsigma$ contains both $\widetilde{\gamma_0}$ and  $\widetilde{\gamma_0}'$; see Figure \ref{curvesigma}.

We denote by $\rho$ the free homotopy class of loops in $S$ which contain the projection of $\varsigma$ to $S$. The class $\rho$ has the following properties:
    \begin{enumerate}
        \item Curves in $\rho$ are not contractible in $S$,
        \item curves in $\rho$ are contractible in $T^2$,
        \item curves in $\rho$ are not homotopic to $\partial S$ in $S$.
    \end{enumerate}
Given any $\varepsilon>0$, there are curves in $\rho$ with length $\leq 2(l_{g_0}(\gamma_0) + \sqrt{D(g_0)})+ \varepsilon$.
This is obtained by considering curves homotopic to $\varsigma$ and contained in sufficiently small neighbourhoods of $\widetilde{\gamma_0} \cup \widetilde{\gamma_0}'\cup \iota$; see Figure \ref{curvesigma}.
 
\begin{figure}
  \includegraphics[width=11cm]{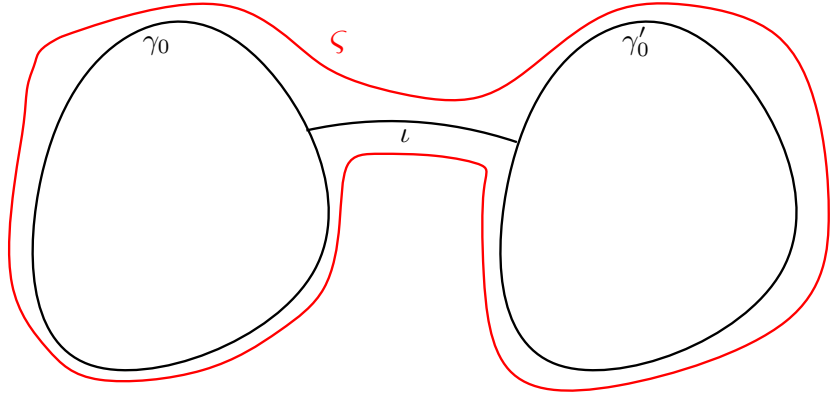}
  \caption{The construction of the curve $\varsigma$.}
  \label{curvesigma}
\end{figure} 
 
We define
$$\sigma_\rho := \inf_{\gamma\in\rho}\Ecal_{g_0}(\gamma).$$
By the remark above we know that $\sigma_\rho \leq  2(l_{g_0}(\gamma_0) + \sqrt{D(g_0)}) $
The infimum $\sigma_\rho$ is actually a minimum, and this can be proved using piecewise smooth geodesics and the strategy to prove \cite[Theorem 1.5.1]{Jost}.
Let $\Acal_\rho$ be the set of curves in $S$ whose $g_0$-energy is $\sigma_\rho$. They are all smooth geodesics of $g_0$.

\

{\bf Step 3: Repelling boundary} 
We show that energy minimising curves have to be contained in the interior of $S$.
    \begin{lemma}\label{lem:intersectingboundary}
    Given $\rho$ as in the last step, there is $\delta_\rho>0$ such that every loop in $\rho$ that intersects $\partial S$ has energy $>\sigma_\rho+\delta_\rho$. 
    \end{lemma}
    \begin{proof}
        If this is false, then we find a sequence $\tau_m$ of $C^\infty$ loops in $\rho$ that touch $\partial S$ with $\Ecal_{g_0}(\tau_m)\to\sigma_\rho$. Then it is possible to find a natural number $K$ and replace $\tau_m$ by a sequence of piecewise smooth geodesics $\haat\tau_m$ in $\rho$ with at most $K$ corners and satisfying $\Ecal_{g_0}(\haat\tau_m)\to\sigma_\rho$. \footnote{To guarantee that we can find a uniform bound on the number of corners of elements of the sequence one uses that  $(S,g_0)$ has geodesic boundary and that its convexity radius is therefore positive.} Then, $\haat\tau_m$ converges up to subsequence to a piecewise smooth geodesic $\tau$ that touches the boundary and that has energy $\Ecal_{g_0}(\tau)=\sigma_\rho$. Because $\Ecal_{g_0}(\tau)=\sigma_\rho$ it must be a smooth geodesic and since it is contained in $S$ and touches $\partial S$ it must be tangent to $\partial S$. This implies that $\tau$ is a geodesic tangent to the geodesic $\partial S$ and thus $\tau$ coincides with $\partial S$. But this is impossible because curves in $\rho$ are not homotopic to the boundary.
    \end{proof}

{\bf Step 4: Perturbing the metric.}
    Consider now the open set $V_{\delta_\rho/2}$ of loops contained in $S$ and belonging to $\rho$ with energy in the interval $[\sigma_\rho,\sigma_\rho+\delta_\rho/2)$. From Step $3$, all loops in $\overline{V_{\delta_\rho/2}}$ have image in the interior of $S$. 
    
    Denote by $U_{\varepsilon}$ the set of Riemannian metrics $g$ satisfying $(1-\varepsilon)g_0<g<(1+\varepsilon)g_0$. 
    For $g\in U_\varepsilon$ and any $\tau' \in W^{1,2}(S^1,T^2)$ we have 
    $$(1-\varepsilon)E_{g_0}(\tau)<E_g(\tau')<(1+\varepsilon)E_{g_0}(\tau).$$
    Choose $\varepsilon>0$ small enough such that
    \begin{equation}
        (1+\varepsilon)\sigma_\rho < (1-\varepsilon)(\sigma_\rho + \frac{\delta_\rho}{2})
    \end{equation}
    
    Then, we can show that any $g\in U_\varepsilon$ has a closed contractible geodesic which belongs to $\overline{V_{\delta_\rho/2}}$. For this we reason as follows:
    
    Let $\gamma_\rho$ be a geodesic of $g_0$ in the class $\rho$ and with $E_{g_0}(\gamma_\rho)=\sigma_\rho$. The curve $\gamma_\rho$ is in the interior of $\overline{V_{\delta_\rho/2}}$. We have $$E_g(\gamma_\rho)<(1+\varepsilon)\sigma_\rho.$$
    Let $\phi_g$ be the negative gradient flow of $\Ecal_g$. Then $\phi_g^t(\gamma_\rho)\in \overline{V}_{\delta_\rho/2}$ for all $ t\geq 0$ because $E_g$ decreases along trajectories of $\phi_g$ and because for all $\tau' \in \partial \overline{V}_{\delta_\rho/2}$ we have
    $$E_g(\tau')\geq (1-\varepsilon)E_{g_0}(\tau') = (1-\varepsilon)(\sigma_\rho+\delta_\rho/2)\geq (1+\varepsilon)\sigma_\rho>E_g(\gamma_\rho).$$
    It follows that $\phi^t_g(\gamma_\rho)$ cannot cross the boundary of $\overline{V_{\delta_\rho/2}}$ and is thus trapped in $\overline{V_{\delta_\rho/2}}$. By Palais-Smale $\phi^t_g(\gamma_\rho)$ converges to a geodesic $\gamma$ of $g$ which must then be in $\overline{V_{\delta_\rho/2}}$. The geodesic $\gamma$ has image in the interior of $S$ and belongs to $\rho$. 
\\\    
    
{\bf Step 5: Uniform lower bound on the entropy.} To obtain a uniform lower bound on  the topological entropy of metrics in $U_\varepsilon$ one uses an argument with finite covers identical to the one used in the proof  of the non-degenerate case. The crucial point is that the Riemannian metric $g$ has a contractible closed geodesic whose length is very close to $\sigma_\rho \leq 2(l_{g_0}(\gamma_0) + \sqrt{D(g_0)}) $.  

\end{proof}

\subsection{A lower bound on topological entropy in terms of surface area}\label{subsec:areabound}
The lower bound on the topological entropy obtained in the proofs of the first part of Theorem~\ref{thm:robustDM} depends on the number of fundamental domains a lift of a closed contractible geodesic in the torus $(T^2,g)$ of a fixed length may intersect.  
A metric invariant that often appears naturally in the investigation of lower bounds on entropy is the surface area.  In this section we observe that there is a lower bound depending only on surface area and minimal length of closed contractible geodesics on $(T^2,g_0)$.
 Since surface area $C^0$ continuously depends on the metrics on $T^2$ this gives an alternative robust bound.
 
The main additional idea is to first find for a genus $1$ surface $(\Sigma, g)$ with one (geodesic) boundary component $\gamma$ two short curves $u_+$ and $u_-$ that generate exponential growth rate $\log(2)$ in the fundamental group of $\Sigma$, where the length of $u_+$ and $u_-$ are bounded in terms the length of $\gamma$ and the area of $(\Sigma, g)$.


Precisely, let $\Sigma$ be a surface homeomorphic to a torus with boundary. We endow $\Sigma$ with a Riemannian metric $g$ such that the boundary curve is a geodesic.  This is in particular the situation given by a closed geodesic on a torus, of which we remove the simply  connected component of the complement of the geodesic. 
We fix an orientation of $\Sigma$ which induces an orientation on $\partial \Sigma$.
Denote by $L$ the length of $\partial \Sigma$ and by $A$ the area of $(\Sigma, g)$. 



We consider the following set of curves:
\[B=\left\lbrace \beta : I\to \Sigma \mbox{ with } \beta(\partial I)\subset \partial \Sigma \mbox{ and } \left[\beta\right] \mathrm{non}\; \mathrm{trivial}\; \mathrm{and}\; \mathrm{non}\; \mathrm{homotopic}\; \mathrm{to}\;\partial\Sigma\right\rbrace\]
and the quantity
\[d=\inf\left\lbrace l_g(\beta)\;,\;\beta\in B\right\rbrace.\]
The infimum is obtained. Let $\alpha$ be a curve in $B$ such that $l_g(\alpha)= d$.
Let $x\in\Sigma$ be the point of $\alpha$ that divides it in two paths of equal length $d/2$,
$\alpha_-$ the restriction of $\alpha$ to the path between $\alpha(0)$ and $x$; $\alpha_+$ the restriction of $\alpha$ to the path between $x$ and $\alpha(1)$.

Let us denote also by $\gamma_-$ and $\gamma_+$ two paths in $\partial \Sigma$ (disjoint up to the endpoints) with $\gamma_-(0) = \gamma_+(0) = \alpha(1)$ and $\gamma_-(1) = \gamma_+(1) = \alpha(0)$ (one of which is constant if the endpoints of $\alpha(0) = \alpha(1)$ are the same). We choose these path such that $\gamma_+$ follows $\partial \Sigma$ in positive orientation and $\gamma_-$ follows  $\partial\Sigma$ in negative orientation. 

We consider the two homotopy classes $\xi_\pm\in \pi_1(\Sigma,x)$:
\[\xi_\pm=\left[ \alpha_+\cdot \gamma_\pm\cdot \alpha_-\right].\]

We need to argue that, first, we can find representatives of $\xi_\pm$ of small lengths and, second, that $\xi_\pm$ generate exponential growth in $\pi_1(\Sigma,x)$.

\textbf{Representatives of small lengths.} The general argument is to consider a family of closed curves such that each one together with $\partial \Sigma$ bounds an annulus in $\Sigma$. Integration on their lengths will show that there is an upper bound on the lengths of smallest representatives of $\xi_{\pm}$ in terms of $A$ and $L$. 

\begin{lemma}\label{lem:bound_generators}
There exist representatives of $\xi_\pm$, both of lengths smaller than $M:=\sqrt{L^2+4A}$. 
\end{lemma}

\begin{proof}

We introduce the equidistant sets for $t<\frac{d}{2}$,
\[\mathcal{G}^t = \left\lbrace x\in \Sigma; \dist_g(x,\partial \Sigma)=t\right\rbrace.\]

It follows from a classical result of Hartman \cite{Hartman}, see also \cite[Theorem 4.4.1]{ShiohamaShioyaTanaka}, that there is a closed set of zero Lebesgues measure of exceptional parameters $E \subset [0,d/2]$ such that for all $t\in R := [0,d/2] \setminus E$, $\mathcal{G}^t$ consists of finitely many connected components, each beeing the image of a piecewise smooth closed simple path. In particular, the length $l_g(\mathcal{G}^t)$ of $\mathcal{G}^t$ is well-defined for $t\in R$.
 Moreover, if we denote $\mathcal{B}^t := \left\lbrace x\in \Sigma; \dist_g(x,\partial \Sigma)\leq t\right\rbrace$, which is a subsurface in $\Sigma$ bounded by $\partial \Sigma$ and $\mathcal{G}^t$, then for $t\in R$, $\frac{d}{dt} \area_g(\mathcal{B}^t) = l_g(\mathcal{G}^t)$.
 
It is straightforward to see that $\alpha$ is not contained in $\mathcal{B}^t$. 
Note also that, if $t< d/2$, $\mathcal{B}^t$ has genus zero. Take the collection of discs in $\Sigma$ that any component of $\mathcal{G}^t$ may bound in the complement of $\mathcal{B}^t$ and attach it to $\mathcal{B}^t$. We denoted this surface  by $\widetilde{\mathcal{B}^t}$. 
\begin{claim}\label{claim_ann}
$\widetilde{\mathcal{B}^t}$ is an annulus that is bounded by $\partial \Sigma$ and one distinguished connected component $\gamma^t$ of $\mathcal{G}^t$.
\end{claim}
\textit{Proof:}
Assume by contradiction that the boundary of $\widetilde{\mathcal{B}^t}$ has more than one connected component besides $\partial \Sigma$. Then there exist a path $\beta$, completely contained in ${\mathcal{B}^t}$,  with endpoints in $\partial\Sigma$, and which is non-homotopic in $\Sigma$ relative to $\partial \Sigma$ to a  path in $\partial \Sigma$. 
Let $l_0$ be the infimum of the lengths of such paths. We claim that $l_0<d$ which contradicts the definition of $d$.
For this, choose $\beta$ as above to have length $l_0 \leq l_g(\beta) < l_0 + (d-2t)$. 
There is a point $y$ on $\beta$ that divides $\beta$ into two subpaths $\beta_0$ and $\beta_1$  of equal length $l_g(\beta)/2$. By the definition of $\mathcal{B}^t$, there is a path $\hat{\beta}$ from $y$ to $\partial \Sigma$ in $\mathcal{B}^t$ of length $\leq t$. Either the concatenation of $\beta_0 \cdot \hat{\beta}$ or $\overline{\beta_1} \cdot \hat{\beta}$ is non-homotopic relative $\partial \Sigma$ to a path in $\partial \Sigma$. Hence $l_0\leq l_g(\beta)/2 + t < l_0/2  + d/2$, and hence $l_0 < d$. 
\qed



Via the annulus that $\gamma^t$ and $\partial \Sigma$ bound we choose the orientation of $\gamma^t$ to be parallel to that of $\partial \Sigma$.   
For each $t\in R$, $t< d$, both paths,  $\alpha_+$ and $\alpha_-$, intersect $\gamma^t$ in some point, say $z^t_+$ and $z^t_-$, respectively. 
Let $\alpha^t_-$ be the subpath of $\alpha_-$ from $z^t_-$ to $x$, and  $\alpha^t_+$ be the subpath of $\alpha_+$ from $x$ to $z^t_+$.
Let $\gamma_+^t$ and $\gamma_-^t$ be two path on $\gamma^t$  (disjoint up to the endpoints) with $\gamma_+^t(0) = \gamma_-^t(0) = z^t_+$ and $\gamma_+^t(1) = \gamma_-^t(1) = z^t_-$.  We choose these path such that $\gamma_+^t$ follows $\gamma^t$ in positive orientation and $\gamma_-^t$ follows  $\gamma^t$ in negative orientation. 
Consider the loops $u^t_\pm=\alpha^t_+ \cdot \gamma^t_\pm \cdot \alpha^t_-$.
By the choice of orientations we have that $[u^t_\pm] = \xi_{\pm} \in \pi_1(\Sigma,x)$. 
The point $z^t_{\pm} \in \gamma^t$ are at distance $t$ from $\partial\Sigma$, by the definition of $\gamma^t$. Hence, the length of $u^t_\pm$ satisfy
\[l_g(u^t_\pm) \leqslant d -2t + l_g(\gamma^t).\]

We will show that there is $t\in R$ such that the length of both $u^t_{\pm}$ is $<M$. (Note here, that also $0\in R$, and $l_g(u^0_{\pm}) < L + d$). 
Assume the contrary. This means that for all $t\in R$, \[l_g(\gamma^t)\geqslant M+2t-d.\]

 Now, for  any $0\leq \sigma \leq d/2$, 
\begin{align}
\begin{split}
A &\geqslant \area_g(\mathcal{B}^{d/2}) - \area_g(\mathcal{B}^{d/2 -\sigma}) \geqslant 
\int_{d/2 -\sigma}^{d/2} l_g(\gamma^t) \, dt\\ &\geqslant \int_{d/2 -\sigma}^{d/2} M +2t-d \, dt  = \int_0^{\sigma} M-2s\,  ds = M\sigma -\sigma^2, 
\end{split}
\end{align}
hence for all $0\leq \sigma \leq d/2$, 
$$0 \leqslant \sigma^2- M\sigma + A = \left(\sigma - \frac{1}{2}(M - \sqrt{M^2 -4A})\right)\left(\sigma - \frac{1}{2}(M+\sqrt{M^2-4A})\right),$$ which means that 
$$M - \sqrt{M^2 -4A} > d,$$ so  
${M} - L > d$, and therefore
$M> d+ L \geq l_g(u^0_{\pm})$, a contradiction to our assumption. 
\end{proof}
\textbf{Generating growth.}
With Lemma \ref{lem:bound_generators} we obtain   
\begin{lemma}\label{lem:generating_growth}
The number of free homotopy classes of loops in $\Sigma$ that have a representative of length $\leq Mn$ is $\geq \frac{2^n-2}{n}$. 
\end{lemma}

\begin{proof}
$\pi_1(\Sigma, x)$ is a free group of two generators. It is straightforward to check, e.g. with the help of the loops $\alpha$ and $\beta$ of Figure~\ref{fig:oneholedtorus},  that we can choose two elements $a, b \in \pi_1(\Sigma, x)$ that freely generate $\pi_1(\Sigma, x)$ such that $\xi_- = a$ and $\xi_+ = b^{-1} a b$. For a given word $w$ of length $n$ in the letters $\xi_+$ and $\xi_-$, consider the word $\widetilde{w}$ in the letters $a$ and $b$ that we obtain if expressing $\xi_-$ as $a$ and $\xi_+$ as $b^{-1} a b$ and then reducing  cyclically. 
It is straightforward to check that,  if $w_1 \neq w_2$ are such words that are not of the form $\xi_+^n$ or $\xi_-^n$, then $\widetilde{w}_1 \neq \widetilde{w}_2$. The homotopy classes of free loops in $\Sigma$ correspond to conjugacy classes of elements in $\pi_1(\Sigma, x)$. The latter correspond to words in $a$ and $b$ up to cyclic reduction and cyclic permutation. 
Hence by Lemma \ref{lem:bound_generators} and the above considerations, the number of homotopy classes of loops in $\Sigma$ that have a representative of length $\leq Mn$ is at least $\frac{2^n-2}{n}$. 
\end{proof}

With Lemmas~\ref{lem:bound_generators} and \ref{lem:generating_growth} we prove the second part of Theorem~\ref{thm:robustDM}. 

\begin{proof}
Assume first that $g_0$ is bumpy. Then, as above,  there is for all $\varepsilon>0$ some $\delta >0$ such that for all $g$ with $d_{C^0}(g,g_0) < \delta$ a closed contractible geodesic $\gamma$ for $g$  with length $l_{g}(\gamma) < l_{g_0}(\gamma_0) + \varepsilon$. Furthermore, by a sufficiently small choice of $\delta$, one can additionally assume that $\area_{g}(T^2) < \area_{g_0}(T^2) + \varepsilon$. 
Therefore in this case it is enough to prove the lower bound for $h_{top}(\varphi_{g_0})$.

Denote $G$ the group of deck transformations for the covering $\Pi: \widetilde{T^2} \to T^2$, and choose a lift $\widetilde\gamma_0:S^1 \to \widetilde{T^2}$ of $\gamma_0$.

We distinguish two cases: 
\begin{enumerate}
\item For all $\mathcal{T} \in G \setminus \{\id\}$, $\mathcal{T}(\im(\widetilde\gamma_0)) \cap \im(\widetilde\gamma_0) = \emptyset$, 
\item There is $\mathcal{T} \in G \setminus \{\id\}$, $\mathcal{T}(\im(\widetilde\gamma_0)) \cap \im(\widetilde\gamma_0) \neq \emptyset$. 
\end{enumerate}
If case $(1)$ holds, then, as in Step 2 of the proof presented in section \ref{subsec:nondeg} for item (1) of Theorem \ref{thm:robustDM}, we know that there is actually a simple closed contractible geodesic $\gamma$ for $g_0$ with length $l_{g_0}(\gamma) \leq l_{g_0}(\gamma_0)$. Denote $D \subset T^2$ the disc that is bounded by $\gamma$. 
 Then, by Lemma \ref{lem:generating_growth} applied to $\Sigma = T^2\setminus D$, and Lemma \ref{lem:nonhomotopicGeodesicsToEntropy}, we obtain that $$h_{top}(\varphi_{g_0}) > \frac{1}{\sqrt{4\area_{g_0}(\Sigma) + l_{g_0}(\gamma)^2}} \log(2) >\frac{1}{\sqrt{4A + l_{g_0}(\gamma_0)^2}}\log(2).$$ 

In case $(2)$, fix $\mathcal{T} \in G$, $\mathcal{T} \neq \id$ with $\mathcal{T}(\im(\widetilde\gamma_0)) \cap \im(\widetilde\gamma_0) \neq \emptyset$, and let $k\in \N$, $k\geq 2$, with $k= \min\{ l\in \N \, |\,  \mathcal{T}^l(\im(\widetilde\gamma_0)) \cap \im(\widetilde\gamma_0) =  \emptyset\}$. 
Consider the lifts $\widetilde\gamma_1 = \mathcal{T} \circ \widetilde\gamma_0$ and $\widetilde\gamma_k = \mathcal{T}^k \circ \widetilde\gamma_0$. 
Then $\widetilde\gamma_1$ intersects both $\widetilde\gamma_0$ and $\widetilde\gamma_k$, hence $\dist_{\Pi^*g_0}(\im(\widetilde\gamma_0), \im(\widetilde\gamma_k)) < l_{g_0}(\gamma_0)/2$.

Choose $\mathcal{S} \in G$ such that for all $l \in \Z \setminus \{0\}$, and all $m \in \Z$, $\mathcal{S}^l(\im(\widetilde\gamma_0)) \cap   \mathcal{T}^m(\im(\widetilde\gamma_0)) = \emptyset$. Denote $\widehat{T^2}$ the quotient of $\widetilde{T^2}$ by the action of the subgroup of $G$ generated by $\mathcal{S}$ and $\mathcal{T}^k$, and $\widehat{\Pi}:\widehat{T^2} \to T^2$ the induced covering map.
 By the choice of $\mathcal{S}$, $\mathcal{T}$ and $k$, the projected curve  $\widehat{\gamma}_0$ of  $\widetilde\gamma_0$ to $\widehat{T^2}$ is a closed geodesic of $\widehat{\Pi}^*g_0$ that lies in one fundamental domain of the universal covering of $\widehat{T^2}$.
So, as argued before, if $\widehat{\gamma}_0$ is not simple, we may replace it by a closed simple geodesic on $(\widehat{T^2}, \widehat{\Pi}^*g_0)$  with length $\leq l_{g_0}(\gamma_0)$ that encircles $\im(\widehat{\gamma}_0)$. Let $\widehat{D}$ be the disc bounded by $\widehat{\gamma}_0$, let $\widehat{\Sigma}= \widehat{T^2}\setminus \widehat{D}$.  
With analogous notation as before, 
\[\widehat{B}:=\left\lbrace \beta : I\to \Sigma \mbox{ with } \partial I\subset \partial \widehat{\Sigma} \mbox{ and } \left[\beta\right] \mathrm{non}\; \mathrm{trivial}\; \mathrm{and}\; \mathrm{non}\; \mathrm{homotopic}\; \mathrm{to}\;\partial\widehat{\Sigma}\right\rbrace,\] one has 
$\widehat{d}:=\inf\left\lbrace l_{\widehat{\Pi}^*g_0} (\beta)\;,\;\beta\in \widehat{B}\right\rbrace \leq \dist_{\Pi^*g_0}(\im(\widetilde\gamma_0), \im(\widetilde\gamma_k)) < l_{g_0}(\gamma_0)/2$.  
Then, following the construction in the proof of Lemma \ref{lem:bound_generators} and Lemma \ref{lem:generating_growth} one proves that  
$$h_{top}(\widehat{\varphi}_{g_0}) > \frac{1}{\frac{3}{2}L} \log(2),$$
where $\widehat{\varphi}_{g_0}$ denotes the geodesic flow on $\widehat{T^2}$ with respect to the metric $\widehat{\Pi}^*g_0$.  
Since $h_{top}(\widehat{\varphi}_{g_0}) = h_{top}(\varphi_{g_0})$, the assertion follows. 

In the case that $g$ is degenerate, one combines the argument for the proof of the degenerate case of the first part of the theorem with the estimates obtained in the previous paragraph, and observes that there is a locally energy minimizing closed contractible geodesic  $\gamma$ whose lift to the universal cover encircles two distinct lifts of $\gamma_0$ and whose length satisfies 
$$l_{g_0}(\gamma) < \max\left\{4\sqrt{4\area_{g_0}(T^2) + l_{g_0}(\gamma_0)^2},3l_{g_0}(\gamma_0)\right\}.$$
The estimates for $h_{\topo}(g)$ stated in the theorem follow as in the previous paragraph. 
\end{proof}


\section{Robustness of volume entropy and hyperbolic geometry.}\label{sec:hyperbolic}

\definecolor{ttzzqq}{rgb}{0.2,0.6,0}
\definecolor{qqzzqq}{rgb}{0,0.6,0}
\definecolor{ccqqqq}{rgb}{0.8,0,0}
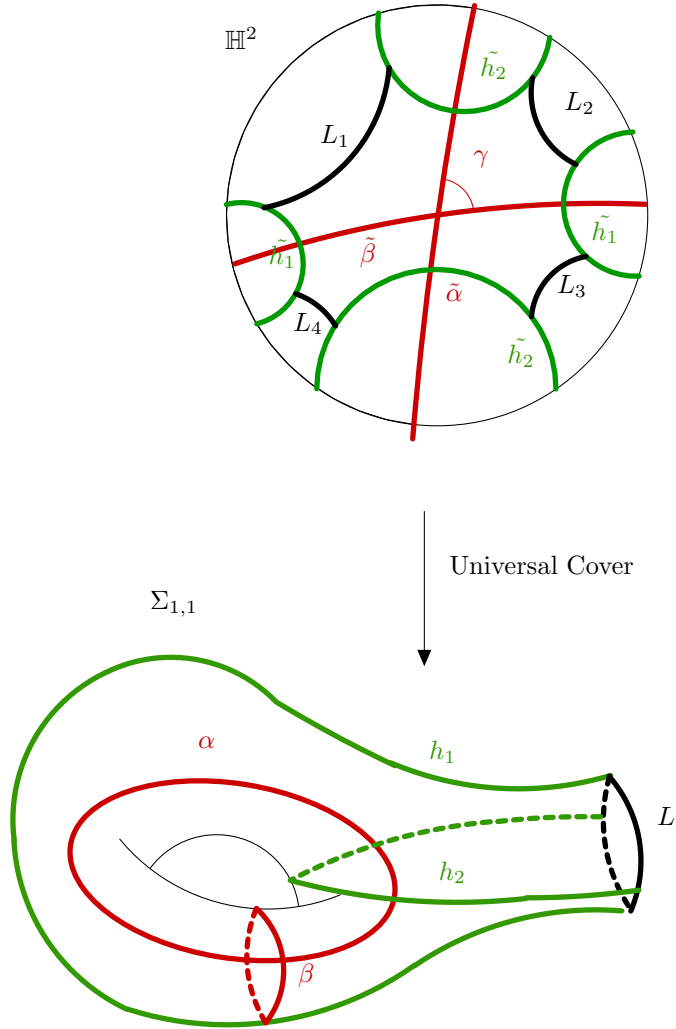
\begin{figure}[!h]
\begin{tikzpicture}[line cap=round,line join=round,>=triangle 45,x=1.0cm,y=1.0cm]
\clip(-4.79,-8.45) rectangle (16.53,7.5);
\draw(2.94,3.37) circle (2.8cm);
\draw [shift={(2.94,3.37)}] plot[domain=1.4:4.6,variable=\t]({1*2.8*cos(\t r)+0*2.8*sin(\t r)},{0*2.8*cos(\t r)+1*2.8*sin(\t r)});
\draw [shift={(51.1,-3.58)},line width=2pt,color=ccqqqq]  plot[domain=2.94:3.06,variable=\t]({1*48.65*cos(\t r)+0*48.65*sin(\t r)},{0*48.65*cos(\t r)+1*48.65*sin(\t r)});
\draw [shift={(5.1,-11.47)},line width=2pt,color=ccqqqq]  plot[domain=1.53:1.9,variable=\t]({1*15*cos(\t r)+0*15*sin(\t r)},{0*15*cos(\t r)+1*15*sin(\t r)});
\draw [shift={(5.58,3.52)},line width=2pt,color=qqzzqq]  plot[domain=1.61:4.76,variable=\t]({1*0.96*cos(\t r)+0*0.96*sin(\t r)},{0*0.96*cos(\t r)+1*0.96*sin(\t r)});
\draw [shift={(3.3,5.9)},line width=2pt,color=qqzzqq]  plot[domain=3:6.14,variable=\t]({1*1.15*cos(\t r)+0*1.15*sin(\t r)},{0*1.15*cos(\t r)+1*1.15*sin(\t r)});
\draw [shift={(0.34,2.72)},line width=2pt,color=qqzzqq]  plot[domain=-1.32:1.82,variable=\t]({1*0.82*cos(\t r)+0*0.82*sin(\t r)},{0*0.82*cos(\t r)+1*0.82*sin(\t r)});
\draw [shift={(2.93,1.06)},line width=2pt,color=qqzzqq]  plot[domain=0:3.14,variable=\t]({1*1.59*cos(\t r)+0*1.59*sin(\t r)},{0*1.59*cos(\t r)+1*1.59*sin(\t r)});
\draw [->] (2.77,-0.57) -- (2.77,-2.63);
\draw [shift={(0.23,5.51)},line width=2pt]  plot[domain=4.91:6.2,variable=\t]({1*2.08*cos(\t r)+0*2.08*sin(\t r)},{0*2.08*cos(\t r)+1*2.08*sin(\t r)});
\draw [shift={(5.25,5)},line width=2pt]  plot[domain=2.95:4.26,variable=\t]({1*1.06*cos(\t r)+0*1.06*sin(\t r)},{0*1.06*cos(\t r)+1*1.06*sin(\t r)});
\draw [shift={(5.09,1.94)},line width=2pt]  plot[domain=1.76:3.05,variable=\t]({1*0.9*cos(\t r)+0*0.9*sin(\t r)},{0*0.9*cos(\t r)+1*0.9*sin(\t r)});
\draw [shift={(0.74,1.39)},line width=2pt]  plot[domain=0.54:1.24,variable=\t]({1*0.99*cos(\t r)+0*0.99*sin(\t r)},{0*0.99*cos(\t r)+1*0.99*sin(\t r)});
\draw [shift={(4.02,0.07)},line width=2pt,color=ttzzqq]  plot[domain=4.31:5,variable=\t]({1*4.33*cos(\t r)+0*4.33*sin(\t r)},{0*4.33*cos(\t r)+1*4.33*sin(\t r)});
\draw [shift={(9.97,12.07)},line width=2pt,color=ttzzqq]  plot[domain=4.17:4.27,variable=\t]({1*17.73*cos(\t r)+0*17.73*sin(\t r)},{0*17.73*cos(\t r)+1*17.73*sin(\t r)});
\draw [shift={(-0.6,-4.46)},line width=2pt,color=ttzzqq]  plot[domain=0.77:2.03,variable=\t]({1*1.95*cos(\t r)+0*1.95*sin(\t r)},{0*1.95*cos(\t r)+1*1.95*sin(\t r)});
\draw [shift={(-0.53,-4.67)},line width=2pt,color=ttzzqq]  plot[domain=2.02:3.27,variable=\t]({1*2.17*cos(\t r)+0*2.17*sin(\t r)},{0*2.17*cos(\t r)+1*2.17*sin(\t r)});
\draw [shift={(-0.05,-4.81)},line width=2pt,color=ttzzqq]  plot[domain=3.19:4.27,variable=\t]({1*2.64*cos(\t r)+0*2.64*sin(\t r)},{0*2.64*cos(\t r)+1*2.64*sin(\t r)});
\draw [shift={(0.15,-3.09)},line width=2pt,color=ttzzqq]  plot[domain=4.4:5.33,variable=\t]({1*4.31*cos(\t r)+0*4.31*sin(\t r)},{0*4.31*cos(\t r)+1*4.31*sin(\t r)});
\draw [shift={(5.05,-10.15)},line width=2pt,color=ttzzqq]  plot[domain=1.49:2.17,variable=\t]({1*4.28*cos(\t r)+0*4.28*sin(\t r)},{0*4.28*cos(\t r)+1*4.28*sin(\t r)});
\draw [shift={(0.69,-3.31)}] plot[domain=3.83:5.1,variable=\t]({1*2.55*cos(\t r)+0*2.55*sin(\t r)},{0*2.55*cos(\t r)+1*2.55*sin(\t r)});
\draw [shift={(0.01,-5.97)}] plot[domain=0.14:2.5,variable=\t]({1*1.1*cos(\t r)+0*1.1*sin(\t r)},{0*1.1*cos(\t r)+1*1.1*sin(\t r)});
\draw [shift={(3.94,-5.22)},line width=2pt]  plot[domain=-0.4:0.72,variable=\t]({1*1.71*cos(\t r)+0*1.71*sin(\t r)},{0*1.71*cos(\t r)+1*1.71*sin(\t r)});
\draw [shift={(7.26,-4.68)},line width=2pt,dash pattern=on 3pt off 3pt]  plot[domain=2.86:3.75,variable=\t]({1*2.11*cos(\t r)+0*2.11*sin(\t r)},{0*2.11*cos(\t r)+1*2.11*sin(\t r)});
\draw [shift={(-0.23,-6.68)},line width=2pt,color=ccqqqq]  plot[domain=-0.66:0.82,variable=\t]({1*1.13*cos(\t r)+0*1.13*sin(\t r)},{0*1.13*cos(\t r)+1*1.13*sin(\t r)});
\draw [shift={(2.08,-6.49)},line width=2pt,dash pattern=on 3pt off 3pt,color=ccqqqq]  plot[domain=2.75:3.7,variable=\t]({1*1.67*cos(\t r)+0*1.67*sin(\t r)},{0*1.67*cos(\t r)+1*1.67*sin(\t r)});
\draw [rotate around={170.88:(0.22,-5.35)},line width=2pt,color=ccqqqq] (0.22,-5.35) ellipse (2.18cm and 1.16cm);
\draw [shift={(4.15,5.03)},line width=2pt,color=ttzzqq]  plot[domain=4.71:4.85,variable=\t]({1*10.74*cos(\t r)+0*10.74*sin(\t r)},{0*10.74*cos(\t r)+1*10.74*sin(\t r)});
\draw [shift={(3.2,2.76)},line width=2pt,color=ttzzqq]  plot[domain=4.45:4.82,variable=\t]({1*8.53*cos(\t r)+0*8.53*sin(\t r)},{0*8.53*cos(\t r)+1*8.53*sin(\t r)});
\draw [shift={(4.28,-11.79)},line width=2pt,dash pattern=on 3pt off 3pt,color=ttzzqq]  plot[domain=1.66:2.05,variable=\t]({1*7.12*cos(\t r)+0*7.12*sin(\t r)},{0*7.12*cos(\t r)+1*7.12*sin(\t r)});
\draw [shift={(4.9,-14.98)},line width=2pt,dash pattern=on 3pt off 3pt,color=ttzzqq]  plot[domain=1.55:1.69,variable=\t]({1*10.35*cos(\t r)+0*10.35*sin(\t r)},{0*10.35*cos(\t r)+1*10.35*sin(\t r)});
\draw (5.74,-4.38) node[anchor=north west] {$ L $};
\draw (1.26,4.7) node[anchor=north west] {$ L_1 $};
\draw (0,6) node[anchor=north west] {$ \mathbb{H}^2 $};
\draw (4.51,5.13) node[anchor=north west] {$ L_2 $};
\draw (4.41,2.7) node[anchor=north west] {$ L_3 $};
\draw (0.9,2.2) node[anchor=north west] {$ L_4 $};
\draw [color=ccqqqq](2.93,2.59) node[anchor=north west] {$ \tilde{\alpha} $};
\draw [color=ccqqqq](1.79,3.2) node[anchor=north west] {$ \tilde{\beta} $};
\draw [shift={(2.94,3.37)},color=ccqqqq]  plot[domain=0.13:1.42,variable=\t]({1*0.49*cos(\t r)+0*0.49*sin(\t r)},{0*0.49*cos(\t r)+1*0.49*sin(\t r)});
\draw [color=ccqqqq](3.3,4.35) node[anchor=north west] {$ \gamma $};
\draw [color=ccqqqq](-0.36,-3.44) node[anchor=north west] {$ \alpha $};
\draw (-1,-1.5) node[anchor=north west] {$ \Sigma_{1,1} $};
\draw [color=ccqqqq](0.97,-6.46) node[anchor=north west] {$ \beta $};
\draw [color=ttzzqq](2.84,-5.1) node[anchor=north west] {$ h_2 $};
\draw [color=ttzzqq](2.71,-3.5) node[anchor=north west] {$ h_1 $};
\draw [color=ttzzqq](3.39,5.68) node[anchor=north west] {$ \tilde{h_2} $};
\draw [color=ttzzqq](3.75,1.86) node[anchor=north west] {$ \tilde{h_2} $};
\draw [color=ttzzqq](4.87,3.5) node[anchor=north west] {$ \tilde{h_1} $};
\draw [color=ttzzqq](0.58,3.17) node[anchor=north west] {$ \tilde{h_1} $};
\draw (2.98,-1.04) node[anchor=north west] {Universal Cover};
\end{tikzpicture}
\caption{\label{fig:oneholedtorus} One-holed torus, generators, heights and boundary.}
\end{figure}

We consider a torus with a boundary component $\Sigma_{1,1}$, endowed with a hyperbolic metric. Associated to this holed torus is its volume entropy, donoted $h_{\mathrm{vol}}$ (cf.\ Definition \ref{def:hvol}). Note that the universal cover is not the whole of $\mathbb{H}^2$ but only a convex subset of $\mathbb{H}^2$ bounded by the lifts of the boundary curve $L$. Hence $h_{\mathrm{vol}}$ is generally smaller than 1.

Both $h_{\mathrm{vol}}$ and $L$ are considered as functions on the Teichmüller space of $\Sigma_{1,1}$.

We prove the following result

\begin{theorem}
If the volume entropy $h$ tends to 0 in the Teichmüller space, then the boundary length $L$ tends to $\infty$.
\end{theorem}

In other terms, to keep the entropy bounded away from 0, we need to bound from above the boundary length. Note the analogy with Section \ref{sec:DenvirMacKay} for the topological entropy. In this setting, the bound on the area is implicit because any hyperbolic metric on the one holed torus has area $2\pi$ by the Gauss-Bonnet formula

\begin{proof}
We rely on Figure~\ref{fig:oneholedtorus}: we choose $\alpha$ and $\beta$ some generators of the fundamental group $\pi_1(\Sigma_{1,1})= \mathbb{F}_2$, for which we denote by $2a$ and $2b$ the lenghts of the respective geodesic realizations once we fix a hyperbolic metric on $\Sigma_{1,1}$. We also denote by $\gamma$ the angle between the lifts of $\alpha$ and $\beta$.

We consider some \textit{heights} associated to $\alpha$ and $\beta$: they are two geodesic curves starting and ending at the boundary and meeting respectively $\alpha$ and $\beta$ orthogonally.

We then construct a fundamental domain for the action of $\pi_1$ on $\mathbb{H}^2$ as in figure \ref{fig:oneholedtorus}: it's a octagon made with (pieces of) lifts of heights and boundary curve. Note that the lifts of $\alpha$ and $\beta$ can be assumed to meet at their midpoints and that $L_1+L_2+L_3+L_4=2L$. The fundamental domain is a union of 4 pentagons with 1 right angle and one angle $\gamma$ or $\pi-\gamma$.

We may also assume that the two remaining angles in each pentagon is also a right angle. Indeed, doing so decreases the lengths of the $L_i$'s and proving the result for 4 right angles pentagons will yield the result.

For further use, we will set that the remaining angle is $\gamma$ in the pentagons containing as a side $L_2$ and $L_4$ and $\pi-\gamma$ in the two other pentagons 

In this case, $h_{\mathrm{vol}}$ and $L$ as functions of $a$, $b$ and $\gamma$. We will use two facts about those functions.

\begin{enumerate}
\item The entropy $h=h_{\mathrm{vol}}$ satisfies the inequality
\[\frac{1}{1+e^{ha}}+\frac{1}{1+e^{hb}}\leqslant \frac{1}{2}.\]
This is proved in \cite{balacheffmerlin}.
\item The boundary length is expressed as
\[\cosh L_2= \sinh a \sinh b - \cosh a \cosh b \cos \gamma.\]
This formula comes from a trigonometric formula in the pentagon with one side $L_2$, similar formulas allow to express the other $L_i$'s. This formula comes from e.g \cite[p. 37 (iii)]{buserspectra}. As $\gamma\to 0$, this formula implies that $a=b\to\infty$.
\end{enumerate}

We now assume that $h\to 0$ and we want to see that one of the $L_i$'s must tend to $\infty$. First notice that $h$ is a continuous function of $a$, $b$ and $\gamma$. So in order to make $h\to 0$, we need to make $a$, $b$ and $\gamma$ escaping compact sets. Note that $\gamma$ cannot tend to $0$ while keeping $a$ and $b$ bounded (unless $L_2\to \infty$ and we are done) so $a$, $b$ and $\gamma$ escape compact sets if and only if $a$ and $b$ escape compact sets. We argue differently depending on how $a$ and $b$ escape compact sets. Without loss of generality, we assume that $a\leqslant b$. Note that two cases are immediatly excluded because of (1): the case $a\to 0$ and $b\to 0$ and the case $a\to 0$ and $b$ bounded.

\textbf{1st case:} $a$ in bounded and $b\to \infty$. Formula (2) implies that
\[\cosh L_2 \sim \sinh a \frac{e^b}{2}-\cosh a \frac{e^b}{2}\cos\gamma.\]
In order to keep $L_2$ bounded, we need that $\cos\gamma\sim \tanh a$. But the same formula applied to the top left pentagon would implies that in order to keep $L_1$ bounded, we need that $\cos\gamma\sim -\tanh a$ and both are not possible simultaneously.

\textbf{2nd case:} $a\to \infty$ and $b\to \infty$. Then formula (2) implies that
\[\cosh L_2 \sim \frac{e^{a+b}}{4}\left(1-\cos\gamma\right).\]
To keep $L_2$ bounded, we need that $\gamma\to 0$ and this is incompatible with keeping $L_1$ bounded.

\textbf{3rd case:} $a\to 0$ and $b\to \infty$. This is the most subtle case and we need to dig more into formula (1). Formula (2) implies already that
\[\cosh L_2 \sim \frac{a}{2}\frac{e^b}{2}-\frac{e^b}{2}\cos \gamma.\;\;\;\mathrm{ and }\;\;\; \cosh L_1 \sim \frac{a}{2}\frac{e^b}{2}+\frac{e^b}{2}\cos \gamma.\]
Keeping both $L_1$ and $L_2$ bounded would imply that $e^ba$ is bounded, or $b\leqslant c-\log a$, where $c$ is a constant.

On the other hand, formula (1) with $a\to 0$ and $b\to \infty$ becomes
\[b\geqslant \frac{1}{h}\log\left(\frac{4}{ha}\right) + o(a)\]
(see \cite{balacheffmerlin}).
Combining the two formulas, we get
\[c-\log a\geqslant \frac{1}{h}\log\left(\frac{4}{ha}\right) + o(a).\]
Reordering, we have,
\[c+\left(\frac{1}{h}-1\right)\log a\geqslant \frac{1}{h}\log\left(\frac{4}{h}\right) + o(a)\]
which is absurd since the left hand side tends to $-\infty$ while the right hand side tends to $+\infty$.

Another, more intuitive, way to analyse this argument would be to remark that $b\to \infty$ is responsible for the fact that $h\to 0$ and conversely $a\to 0$ has the tendency to keep $h$ away from 0. The opposition is settled by the relation $ae^b$ bounded (which means $a$ wins, $a$ and $b$ dont have a symmetric role). Hence it's reasonable to reach a contradiction if we moreover assume that $h\to 0$.
\end{proof}

We conclude this paragraph by an example showing that we cannot bound from below the volume entropy by an absolute constant. Those examples are well-known, we only discuss them for completeness.

\begin{proposition}
There exists a sequence of hyperbolic metrics on the one-holed torus whose volume entropy tends to zero.
\end{proposition}

\begin{proof}
On the Poincaré disk model for the hyperbolic plane, we consider two orthogonal geodesics $A$ and $B$ meeting at the basepoint $o$. We denote by $\alpha$ and $\beta$ the loxodromic isometries whose axis are $A$ and $B$ respectively and with the same translation length denoted $2a$ (the factor 2 makes the computations a bit easier).

A classical ping-pong type argument shows that, when $a$ is big enough, the group $\Gamma_a$ generated by $\alpha$ and $\beta$ is free, the quotient of the disk by $\Gamma_a$ is a torus with a funnel and the convex core of the later is a torus with one boundary component.

We will argue that, as $a$ tends to infinity, the volume entropy of this torus tends to 0. To achieve this computation, we will use the equality between Hausdorff dimension of the limit set $\Lambda(\Gamma_a)$ and volume entropy \cite{sullivan} and actually compute the Hausdorff dimension.

Since the translation lengths of $\alpha$ and $\beta$ are the same, the limit set is a self-similar Cantor and we may compute its dimension for instance with \cite[thm 4.14]{mattilafractals}. Since $\alpha$ is a Lipschitz map on the boundary, we look at "the quarter of the limit set" given by $\alpha(\Lambda(\Gamma_a))$.

The isometries $\alpha$ and $\beta$ both moove the point $o$ to a point along their axis at Euclidean distance $\tanh a$. We deduce that the contraction ratios of $\alpha$ and $\beta$ on $\alpha(\Lambda(\Gamma_a))$ are
\[r_\alpha=r_\beta=r\frac{\arccos\left(1-\frac{\left(1-\tanh a\right)^2}{2}\right)}{\arccos\left(1-\frac{\left(1-\tanh \left(\frac{a}{2}\right)\right)^2}{2}\right)} \underset{a\to\infty}{\sim}\frac{1+e^a}{1+e^{2a}}.\]
(we use the spherical distance on $\alpha(\Lambda(\Gamma_a))$, it is bi-Lipschitz to the Euclidean distance).
Finally the Hausdorff dimension is given by (\cite[thm 4.14]{mattilafractals})
\[\mathrm{Hdim}(\Lambda(\Gamma_a))=\frac{-\log 3}{\log r}\]
which tends to 0 (linearly) as $a\to\infty$.
\end{proof}

\section{Robustness from intersection patterns of a family of non-contractible geodesics on the two-torus}\label{sec:propertyF}

In this section we discuss how a certain intersection pattern of closed (non-contractible) geodesics on $T^2$ implies robustness of topological entropy. Remarkably, and in contrast to the condition discussed in section \ref{sec:DenvirMacKay}, this intersection pattern appears for a $C^{\infty}$-generic metric. In other words we obtain that topological entropy is $C^0$ robust for $C^{\infty}$-generic metrics. 
The content in this section is motivated by the work of Bolotin and Rabinowitz \cite{BolotinRabinowitz} and Glasmachers and Knieper \cite{GlasmachersKnieper} and we use  some of their results. 

\subsection{A definition of separation for lifts of two freely homotopic loops}
Fix a free homotopy class $\alpha$ of loops in $Q$,  $g\in \mathfrak{G}(Q)$. 
Let $\gamma, \gamma'\in \mathcal{L}_{\alpha}Q$ of energy 
 $a = \mathcal{E}_g(\gamma), a' = \mathcal{E}_g(\gamma')$.
Choose two lifts $\widetilde\gamma$ and $\widetilde\gamma':\R \to \widetilde{Q}$ to the universal cover $\widetilde{Q}$ of $Q$.  (This is understood as first lifting $\gamma$ and $\gamma'$ to maps $\R \to Q$ and then lifting to $\widetilde{Q}$.) 


\begin{definition}
Define $b_0  = b_0(\widetilde\gamma, \widetilde\gamma')$ to be the infimum of the numbers $b>0$ such that there is a continuous path in $\mathcal{L}_{\alpha}Q^{<b}_{g}$  from  $\gamma$ to $\gamma'$ that lifts to a path from $\widetilde\gamma$ to $\widetilde\gamma'$.
We define the \textit{separation} of $\widetilde\gamma$ and $\widetilde\gamma'$ to be the non-negative real number 
$\sep_g(\widetilde\gamma, \widetilde\gamma') = \min\{ \log(\frac{b_0}{a}),\log(\frac{b_0}{a'})\} \geq 0$.
	\flushright$\triangle$
\end{definition}


The following robustness statement follows from the definitions.
\begin{lemma}\label{lem:connected_stable}
Let $\delta>0$, and $g, g'$ be two metrics with 
$d_{C^0}(g,g') < \delta$.
Then $\sep_{g'}(\widetilde\gamma, \widetilde\gamma')\leq \sep_g (\widetilde\gamma, \widetilde\gamma')  + 2\delta$. 
\end{lemma}
\begin{proof}
Let $a=\mathcal{E}_g(\gamma)$ and $a'=\mathcal{E}_g(\gamma')$, and $u_s$ a path from $\gamma$ to $\gamma'$ in $\mathcal{L}Q^{<b}_g$ with $\min(\log(\frac{b}{a}), \log(\frac{b}{a'})) < \sep_g(\widetilde\gamma, \widetilde\gamma') + \epsilon$ for some $\epsilon>0$, and which lifts to a path from $\widetilde\gamma$ to $\widetilde\gamma'$. 
Then $\mathcal{E}_{g'}(\gamma) > e^{-\delta}a$, $\mathcal{E}_{g'}(\gamma') > e^{-\delta}a'$, 
and $\mathcal{E}_{g'}(u_s) < e^{\delta}b$ for all $s$.
In other words $\sep_{g'}(\widetilde\gamma, \widetilde\gamma') < \min\{ \log(\frac{e^{2\delta}b}{a}), \log(\frac{e^{2\delta}b}{a'})\}< \sep_g (\widetilde\gamma, \widetilde\gamma')  + 2\delta + \epsilon$, 
and since $\epsilon>0$ was arbitrary the claim follows. 
\end{proof}

\subsection{The two-torus and an intersection pattern}
In the following let $Q= T^2$ be the $2$-torus, equipped with the standard orientation and $\alpha$ a non-trivial free homotopy class of loops in $T^2$. 
If a lift $\widetilde\gamma$ of a closed oriented curve  $\gamma$ representing $\alpha$ is embedded, then it divides the universal cover $\widetilde{T^2}$ in two connected components, the right $R(\widetilde\gamma)$ and the left $L(\widetilde\gamma)$ of $\widetilde\gamma$. 

In the following definition we formulate an intersection pattern of (lifts of) four closed curves of class $\alpha$, see also Figure \ref{fig:F} below.

\begin{definition}\label{def:F}
We say that four oriented closed curves $\gamma_1, \gamma_2, \gamma_3, \gamma_4$ in $T^2$ that represent $\alpha$  form a \textit{ribbon} if 
\begin{enumerate}
\item[(0)] their lifts to $\widetilde{T^2}$ are embedded,  
\end{enumerate}
and if for some choice of lifts $\widetilde{\gamma}_1, \widetilde\gamma_2, \widetilde\gamma_3, \widetilde\gamma_4$ to $\widetilde{T^2}$,
 \begin{enumerate}
\item $\widetilde\gamma_1$ is on the left of $\widetilde{\gamma_3}$ and $\widetilde{\gamma_4}$; $\widetilde\gamma_{4}$ is on the right of $\widetilde\gamma_1$ and $\widetilde\gamma_2$. 
\item $\widetilde\gamma_1$ and $\widetilde\gamma_2$ intersect, $\widetilde\gamma_2$ and $\widetilde\gamma_3$ intersect, and $\widetilde\gamma_3$ and $\widetilde\gamma_4$ intersect, and all intersections are transverse.
\end{enumerate}

Let $\varepsilon>0$. We say that $\gamma_1, \gamma_2, \gamma_3, \gamma_4$ form an \textit{$\varepsilon$-ribbon} (with respect to the metric $g$),   
if in addition to $(0), (1)$, and $(2)$ (which is included in $(4)$ below) the four lifts $\widetilde\gamma_1, \ldots, \widetilde\gamma_4$  satisfy that
\begin{enumerate}\setcounter{enumi}{2}
\item $\sep(\widetilde\gamma_i, \widetilde\gamma_j) \geq \varepsilon$, for all $i,j \in \{1, \ldots, 4\}$, $i\neq j$.
\item two lifts $\widetilde\tau$ and $\widetilde\tau'$ of two closed curves $\tau, \tau'$ of class $\alpha$ intersect whenever they satisfy one of the following
\begin{itemize} 
\item $\sep_g(\widetilde\tau, \widetilde\gamma_1) < \varepsilon$, and $\sep_g(\widetilde\tau', \widetilde\gamma_2) < \varepsilon$.
\item $\sep_g(\widetilde\tau, \widetilde\gamma_2) < \varepsilon$, and $\sep_g(\widetilde\tau', \widetilde\gamma_3) < \varepsilon$. 
\item $\sep_g(\widetilde\tau, \widetilde\gamma_3) < \varepsilon$, and $\sep_g(\widetilde\tau', \widetilde\gamma_4) < \varepsilon$. 
\end{itemize}
\end{enumerate}
	\flushright$\triangle$
\end{definition}

The following Proposition \ref{prop1} states that an $\varepsilon$-ribbon is robust with respect to the $C^0$ topology on the metrics. The main difficulty is to guarantee items $(0)$ and $(1)$ of a ribbon for a perturbed metric. Here results on the analysis of the curve-shortening flow are used \cite{Angenent1, Angenent2, Grayson}.  

\begin{proposition}\label{prop1}
Assume that there are four curves $\gamma_1, \ldots, \gamma_4$ that form an $\varepsilon$-ribbon for some $\varepsilon>0$ with respect to $g$. Let $\delta>0$ with $\varepsilon> 2\delta >0$. 
Then for any metric $g'$ with $d_{C^0}(g',g) < \delta$ there are four closed geodesics $\gamma'_1, \ldots, \gamma'_4$ that form an $(\varepsilon - 2\delta)$-ribbon with respect to $g'$. 
\end{proposition}

\begin{proof}
The proof uses essentially the analysis of the curve shortening flow. 
Consider the space of embedded closed smooth curves $\Gamma = \{\gamma \in \mathcal{L}Q \, | \, \gamma \text{ embedded}\}$. The curve shortening flow is a continuous local semi-flow $\Phi^t: \Gamma \to \Gamma$,
$\Phi^t(\gamma_0) = \gamma_t$ for $t\in [0,T_{\gamma_0})$, defined by $\frac{\partial \Phi}{\partial t} =k_t N_t$, where $k_t$ is the geodesic curvature of $\gamma_t$ and $N_t$ the unit normal vector. We need the following properties (see \cite{Angenent3}): Mutually non-intersecting curves $\gamma_0$ and $\gamma'_0$ stay non-intersecting along the flow \cite{Angenent1, Angenent2}. Assume that $Q$ is compact with geodesic boundary, then for $\gamma \in \Gamma$, either the maximal $T_{\gamma}$ is finite and $\Phi(\gamma)$ converges to a point, or  $T_{\gamma} = +\infty$ and $\Phi(\gamma)$ converges to a geodesic \cite{Grayson}.
The length is decreasing under $\Phi^t$. 

 Let now $\varepsilon>0$ and $\gamma_1,\ldots, \gamma_4$ be the four curves representing $\alpha$ that form an $\varepsilon$-ribbon with respect to $g$ for lifts $\widetilde{\gamma}_1, \ldots, \widetilde\gamma_4$, and let $\delta>0$ and $g'$ as in the proposition. Let $\mathcal{T}$ be the deck transformation corresponding to $\alpha$ on the universal covering $\widetilde{T^2}$ of $T^2$.
 The quotient $\widetilde{T^2}/\mathcal{T}$ by the action of $\mathcal{T}$ is an annuls, and we apply the curve shortening flow on $\widetilde{T^2}/\mathcal{T}$ with respect to the metric $g'$ starting with the projection on $\widetilde{T^2}/\mathcal{T}$  of the four lifts. The images of the curve shortening flow of these curve stay inside an  annulus with geodesic boundary in $\widetilde{T^2}/\mathcal{T}$.  By the properties mentioned above,  the flow will converge to four embedded geodesics in $\widetilde{T^2}/\mathcal{T}$ that project to four geodesics $\gamma_1', \ldots, \gamma'_4$ in $(T^2, g')$. 
  Moreover, the path of curves given by the curve shortening flow lift to paths from $\widetilde\gamma_i$ to $\widetilde{\gamma}'_i$, $i=1,2,3,4$. These four curves and their lifts form an $(\varepsilon-2\delta)$-ribbon. Properties (0) and (1) follow immediately from the properties of the curve shorting flow. Now, reparametrize uniformly by arc-length without increasing the energy along the four paths of closed curves given by the curve shorting flow, so that also the energy decreases along the obtained path, in particular the energy will be bounded from above by the energy of the starting curve. 
 Hence for $i\neq j$,  $\sep_{g'}(\widetilde\gamma'_i,\widetilde\gamma'_j)\geq \sep_{g'}(\widetilde\gamma_i,\widetilde\gamma_j)$. 
By Lemma \ref{lem:connected_stable}, $\sep_{g'}(\widetilde\gamma_i,\widetilde\gamma_j) \geq \sep_g(\widetilde\gamma_i,\widetilde\gamma_j) - 2\delta \geq \varepsilon- 2\delta$.   
	
  To see property $(4)$, consider two closed curves $\tau$ and $\tau'$  with lifts $\widetilde\tau$ and $\widetilde\tau'$ to $\widetilde{T^2}$ such that $\sep_{g'}(\widetilde\tau, \widetilde\gamma'_1) < \epsilon -2\delta$  and 
	$\sep_{g'}(\widetilde\tau', \widetilde\gamma'_2) < \varepsilon -2\delta$. 
	Since for every $b>\mathcal{E}_g'(\gamma_1)$ there is a path from $\gamma_1$ to $\gamma'_1$ in $\mathcal{L}_{\alpha}Q^{<b}_{\alpha, g'}$,
		$\sep_{g'}(\widetilde\tau, \widetilde\gamma_1) <\varepsilon- 2\delta$, and hence
	$\sep_{g}(\widetilde\tau, \widetilde\gamma_1) <\varepsilon$. Similarly, 	$\sep_{g}(\widetilde\tau', \widetilde\gamma_2) <\varepsilon$.  Therefore, $\widetilde\tau$ and $\widetilde\tau'$ intersect. Similarly one checks the remaining cases of property $(4)$.   
		
\end{proof}

\subsection{Robustness of entropy via ribbons}

We will see that the intersection pattern of geodesics on $T^2$ considered above imply that the metric has robust topological entropy.



\begin{theorem}\label{thm1}
If $(T^2,g)$ admits four closed geodesics  $\gamma_1, \ldots, \gamma_4$ that form a ribbon, then the topological entropy of the geodesic flow $\varphi_g$ is positive. Moreover, the topological entropy is bounded from below by $\frac{1}{L}\log(2)$, where $L= \min\{l(\gamma_1) + l(\gamma_2), l(\gamma_3) + l(\gamma_4)\}$. 
\end{theorem}

\begin{proof}
See  \cite[Lemma 4.2]{GlasmachersKnieper} for a similar argument. 
In the following consider the strip $S = R(\widetilde\gamma_1) \cap L(\widetilde\gamma_4) \subset \widetilde{T^2}$.
By our assumptions, we can choose  $U_0 \subset S$ to be a connected component of $R(\widetilde\gamma_1) \cap L(\widetilde\gamma_2)$   and $U_1 \subset \widetilde{T^2}$ a connected component of  $L(\widetilde\gamma_4) \cap R(\widetilde\gamma_3)$ such that $U_1\cap U_2 \neq \emptyset$. 
Let $\mathcal{T}\neq id$ be the covering transformation corresponding to free homotopy class $\alpha$ of the geodesics $\gamma_1, \ldots, \gamma_4$. 
Note that by the assumptions, for any $i,j \in \Z$, $i\neq j$,  $\mathcal{T}^i U_1 \cap \mathcal{T}^jU_1 =\emptyset$, and $\mathcal{T}^i U_2 \cap \mathcal{T}^j U_2 = \emptyset$. 
For any bi-infinite sequence $\mathfrak{a}= (a_i)_i \in \Z$, $a_i \in \{0,1\}$, consider the set $D(\mathfrak{a}) = S \setminus \bigcup_{i\in\Z} \mathcal{T}^i U_{a_i}$.  For any periodic $\mathfrak{a}$ of period $p$ we choose a closed curve $\gamma$ in class $p\alpha$ which has a lift in $D(\mathfrak{a})$. We can assume that $\gamma$ is a geodesic,  e.g. by applying Lemma \ref{lem:annulus} below,
and has minimal energy, and hence also minimal length, among such curves. In particular, the  length of $\gamma$ is bounded from above by $pL$, where $L= \min\{l(\gamma_1) + l(\gamma_2), l(\gamma_3) + l(\gamma_4)\}$. This follows since there exist both, a closed curve in class $p\alpha$ of length smaller than $p(l(\gamma_1) + l(\gamma_2))$, and one  of length smaller than  $p(l(\gamma_3) + l(\gamma_4))$ whose lifts are contained in the boundary of $D(\mathfrak{a})$.

These geodesics $\gamma$ provide us with separating sets for the geodesic flow of $g$: Let $u$ be compact connected set in $\overline{S} \setminus \bigcup_{i\in\Z} \mathcal{T}^i (U_0 \cap U_1)$  such that $\overline{S} \setminus u$ has two components for which one contains the sets  $\mathcal{T}^i(U_0 \cap U_1), i\leq 0$, and the other the sets $\mathcal{T}^i(U_0 \cap U_1), i>0$. Furthermore, choose two disjoint compact connected sets $v_0 \in U_0 \setminus U_1$ and $v_1 \in U_1 \setminus U_0$ such that $v_0 \cup (U_0 \cap U_1) \cup v_1$ also divide $\overline{S}$ into two components. 
Lift the geodesic flow of $(T^2,g)$ to a flow on $T^1\widetilde{T^2}$ and denote it by $\widetilde{\varphi}_g$, i.e., $\widetilde{\varphi}_g$ is the geodesic flow of $(\widetilde{T^2},\widetilde{g})$ of the lifted metric $\widetilde{g}$. Let $P:T^1\widetilde{T^2} \to T^1{T}^2$ the covering map induced by the universal covering $\widetilde{T^2} \to T^2$. Furthermore, denote by $\widetilde{u}$, $\widetilde{v_0}$, and $\widetilde{v_1}$ the lifts to $T^1\widetilde{T^2}$ of $u,v_0$, and $v_1$, respectively. 
It is easy to see that there is a constant $k>0$, such that for every $t_0>0$ there is a covering of $\bigcup_{t\in [0,t_0]} \widetilde{\varphi}^t_g(\widetilde{u})$ of less than  $kt_0^2$ open sets such that for any of the sets $B$ of the covering and any $x,y \in B$, $d_{\widetilde{g}}(x,y) = d_g(P(x), P(y))$. 

By the discussion above, for any $p$-periodic binary word $\mathfrak{a}=(a_i)_{i\in \Z}$ there is a lift  $\widetilde\gamma$  of a closed geodesic in $(T^2,g)$ of length $\leq pL$ that intersects $u$, that intersects $\mathcal{T}^i(v_j)$ if and only if $a_i = j$, and that intersects $\mathcal{T}^i(U_j)$ if and only if $a_i \neq j$. 
Hence for $\varepsilon$ sufficiently small there is a set $X \subset \widetilde{u}$ of at least $2^{p}$ points such that for all $x,y \in X$, $\sup_{t\in[0,pL]} d_{\widetilde{g}}\left(\widetilde{\varphi}^t_{{g}}(x),\widetilde{\varphi}^t_{{g}}(y)\right)\geq \varepsilon$.
Hence there is an $(\varepsilon,pL)$-separated set of $\varphi_g$ of cardinality at least $2^p/(k(pL)^2)$, and therefore  
$$h_{\topo}(\varphi_g) \geq \limsup_{p\to \infty} \frac{\log(2^p/(k(pL)^2))}{pL} =\frac{1}{L}\log(2).$$

\end{proof}

Theorem \ref{thm1} together with Proposition \ref{prop1} show that the existence of four curves that form an $\varepsilon$-ribbon implies that the topological entropy is positive in a $C^0$ neighborhood of $g$. Naturally, the intersection pattern of a ribbon might not be robust. Nonetheless, the next result asserts that the existence of four geodesics that form a ribbon implies another collection of four geodesics that form an $\varepsilon$-ribbon  for some $\varepsilon>0$, and hence the topological entropy is indeed positive in a $C^0$ neighborhood of $g$.



\begin{theorem}\label{thm2}
If there are four geodesics $\gamma_1, \ldots, \gamma_4$ that form a ribbon for a metric $g$, then there is $\varepsilon>0$ and four geodesics $\eta_1, \ldots, \eta_4$ that form an $\varepsilon$-ribbon for the metric $g$. 
\end{theorem}

From Theorem \ref{thm2}, \ref{thm1} and Proposition \ref{prop1} it follows that
\begin{corollary}
If there are four geodesics $\gamma_1, \ldots, \gamma_4$ that form a ribbon for a metric $g$, then there $\delta>0$ such that for all $g'$ with $d_{C^0}(g',g) < \delta$, the topological entropy of the geodesic flow $\varphi_g'$ of $g'$ is positive.   
\end{corollary}



In the proof of Theorem \ref{thm2} we use the following lemma repeatedly. 
For that let $A$ be an annulus such that each boundary component $b$ of $A$ is piecewise geodesic and such that
on each point on the finite and non-empty set of non-smooth points on $b$ the outer angle is $>\pi$. We say that a  boundary component $b$ 
with that properties is \textit{admissible}. 

\begin{lemma}\label{lem:annulus}
Let $\alpha$ be the free homotopy class of curves in $A$ provided by a choice of generator of $\pi_1(A) = \Z$. 
Then there is a simple geodesic $\gamma$ in $A$ with $\mathcal{E}_g(\gamma) = d := \inf\{\mathcal{E}_g(x) \, | \, x:S^1 \to A, [x]= \alpha\}$ and there is  $\varepsilon>0$ such that $\widetilde{d}  = \inf\{\mathcal{E}_g(x) \, | \, x:S^1 \to A, \, [x]= \alpha,   x \cap \partial A \neq \emptyset\} \geq e^{\varepsilon} d$. 
\end{lemma}
\begin{proof}

The statement of this lemma is well-known and can be proved similarly to Lemma \ref{lem:intersectingboundary}. To that end note that since the boundary components are admissible, any closed piecewise geodesic $\gamma$ in $A$ that intersects a non-smooth point $x$  at some boundary component of $A$ can be replaced by a piecewise geodesic $\gamma'$ in $A$ whose image coincides with that of $\gamma$ outside a neighborhood of $b$ and that has strictly smaller energy.   
\end{proof}

\begin{figure}
\begin{tikzpicture}[scale=0.4]

\begin{scope}\clip (-3,-3.1) rectangle (27,12.1);

\draw[-] (-10,0) -- (30,0) node [pos=0.75,below] {\Large $\widetilde\gamma_4$};
\draw[-] (-10,9) -- (30,9) node [pos=0.75,above] {\Large $\widetilde\gamma_1$};

\draw[-,mark position=0.83(a)] plot [smooth] coordinates {(-12,5) (-8,-3)  (-4,5)  (0,-3)  (4,5)  (8,-3)  (12,5) (16,-3)  (20,5)  (24,-3)  (28,5) (32,-3)}; \node at (a) [right] {\Large $\widetilde\gamma_3$};

\draw[-,mark position=0.83(b)] plot [smooth] coordinates {(-12,4) (-8,12) (-4,4)  (0,12) (4,4)  (8,12)  (12,4) (16,12)  (20,4)  (24,12)  (28,4) (32,12)}; \node at (b) [right] {\Large $\widetilde\gamma_2$};

\draw[dashed,mark position=0.42(c)] plot [smooth] coordinates {(-12,8) (-4,3)   (4,8)  (12,8) (20,3)  (28,8) (36,8)};
\node at (c) [below] {\Large $\widetilde\eta_1$};
\draw[dashed,mark position=0.5(d)] plot [smooth] coordinates {(-12,2) (-4,2)   (4,6)  (12,2) (20,2)  (28,6) (36,2) };
\node at (d) [below] {\Large $\widetilde\eta_3$};
\draw[dashed,mark position=0.25(e)] plot [smooth] coordinates {(-12,6) (-4,1)   (4,1)  (12,6) (20,1)  (28,1) (36,6)};
\node at (e) [above] {\Large $\widetilde\eta_4$};
\draw[dashed,mark position=0.27(f)] plot [smooth] coordinates {(-12,7) (-4,8)   (4,3)  (12,7) (20,8)  (28,3) (36,7)};
\node at (f) [below] {\Large $\widetilde\eta_2$};
\end{scope}

\end{tikzpicture}

\caption{$\widetilde\gamma_1, \widetilde\gamma_2, \widetilde\gamma_3, \widetilde\gamma_4$ satisfy property $(F)$. The paths $\widetilde\eta_1, \widetilde\eta_2, \widetilde\eta_3, \widetilde\eta_4$ are constructed in the proof of Theorem \ref{thm2}, and satisfy property $(F,\varepsilon)$ for some $\varepsilon>0$. }
 \label{fig:F}
\end{figure}
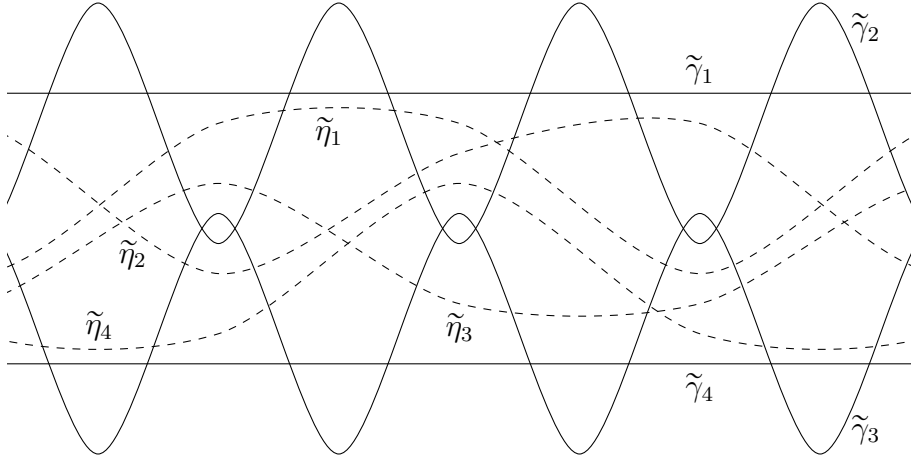

\begin{proof}[Proof of Theorem \ref{thm2}]
Let $\gamma_1, \ldots, \gamma_4$ be four curves with lifts $\widetilde\gamma_1, \ldots, \widetilde\gamma_4$ that form a ribbon. 
Choose $U_0$ and $U_1$ as in the proof before. 
Consider the four bi-infinite sequences $\mathfrak{a}^1, \ldots ,\mathfrak{a}^4$ of period $3$ by extending periodically the four words $110$, $011$, $100$, and $010$, respectively.  
We find in two steps four closed geodesic $\eta_1, \ldots, \eta_4$ representing $3\alpha$ in $T^2$ with four lifts $\widetilde\eta_1,\ldots,  \widetilde\eta_4$ that lie in $D(\mathfrak{a}^1),\ldots, D(\mathfrak{a}^4)$, respectively, and that form an $\varepsilon$-ribbon for some $\varepsilon>0$. 

\textit{Finding $\eta_1$ and $\eta_2$:} Note that  $D(\mathfrak{a}^1)$ and $D(\mathfrak{a^2})$ are invariant under the shift $\mathcal{T}^3$, where $\mathcal{T}$ denotes the shift corresponding to $\alpha$. Hence $D(\mathfrak{a}^1)$ and $D(\mathfrak{a^2})$ project to annuli $A_1$ resp.  $A_2$ in $\widetilde{T^2}/\mathcal{T}^3$. Their boundary components are admissible since the sequences $\mathfrak{a}^1$ and $\mathfrak{a}^2$ are non-constant. 
Fix choices  $\varepsilon_1>0$ and $\varepsilon_2>0$  of the $\varepsilon$ provided by Lemma \ref{lem:annulus} for $A_1$ resp. $A_2$.  
Let $\widehat{\eta}_1$ and $\widehat{\eta}_2$ be energy minimizing geodesics in $A_1$ and $A_2$ respectively, $\widetilde\eta_1, \widetilde{\eta_2}$ some choice of lifts to $\widetilde{T^2}$, and $\eta_1, \eta_2$ their projections to $T^2$. 

\textit{Finding $\eta_3$ and $\eta_4$:} Consider now  $D(\mathfrak{a}^3) \cap R(\widetilde\eta_1)$ and $D(\mathfrak{a}^4)\cap R(\widetilde\eta_1) \cap R(\widetilde{\eta_2})$. These sets are invariant under $\mathcal{T}^3$ and project to annuli $A_3$ resp. $A_4$ in $\widetilde{T^2} / \mathcal{T}^3$. By the choice of the sequences $\mathfrak{a}^1, \ldots, \mathfrak{a}^{4}$, one directly checks that the boundary components of $A_3$ and $A_4$ are admissible. Let $\varepsilon_3>0$ resp.  $\varepsilon_4>0$ be choices of $\varepsilon$ for $A_3$ resp. $A_4$ provided by Lemma \ref{lem:annulus}. Let $\widehat{\eta}_3$ resp.  $\widehat{\eta}_4$ be energy minimizing geodesics in $A_3$ resp.  $A_4$, let $\widetilde\eta_3$, resp.  $\widetilde{\eta}_4$ some choice of lifts to $\widetilde{T^2}$, and $\eta_3$ resp. $\eta_4$ their projections to $T^2$. 

By the intersection properties of the annuli $A_1, A_2, A_3$ and $A_4$, one sees that $\widetilde\eta_1$ and $\widetilde\eta_2$ intersect, that $\widetilde\eta_2$ and $\widetilde\eta_3$ intersect and that $\widetilde\eta_3$ and $\widetilde\eta_4$ intersect. By construction $\widetilde\eta_4$ is on the right of $\widetilde\eta_1$ and $\widetilde\eta_2$, and $\widetilde\eta_3$ is on the right of $\widetilde\eta_1$. Hence the geodesics $\eta_1, \ldots, \eta_4$ form a ribbon. 

Moreover, with $\varepsilon := \min\{\varepsilon_1, \varepsilon_2, \varepsilon_3, \varepsilon_4\}$, $\eta_1, \ldots, \eta_4$ form an $\varepsilon$-ribbon.
Indeed, let us see $\sep_g(\widetilde\eta_1, \widetilde\eta_2)\geq \varepsilon_1$, the remaining  conditions are  checked analogously. 
Choose a path $u:[0,1]\to \mathcal{L}_{3\alpha}T^2$ that lifts to a path $\widehat{u}$ from $\widehat{\eta}_1$ to $\widehat{\eta}_2$.
It is clear that there are $s \in [0,1]$ such that the lifted curve $\widehat{u}(s)$ touches the boundary of $A_1$, and let $s_1>0$ be the infimum in $[0,1]$ of such $s$. By compactness, $\widehat{u}(s_1)$ touches the boundary of $A_1$ and $\widehat{u}(s_1)$ is contained in $A_1$. 
Since $\widehat{\eta_1}$ has minimal energy among all curves of class $3\alpha$ in $A_1$ we conclude with  Lemma \ref{lem:annulus} that $\mathcal{E}_g(\widehat{u}(s_1)) \geq e^{\varepsilon_1}\mathcal{E}_g(\widehat{\eta}_1)$. 
Hence $\sep_g(\widetilde{\eta_1}, \widetilde\eta_2) \geq \log\left(\frac{\mathcal{E}_g(\widehat{u}(s_1))}{\mathcal{E}_g(\widehat{\eta}_1)}\right) \geq \varepsilon_1$. 
\end{proof}

\subsection{Ribbons exist for $C^{\infty}$ generic metrics}\label{sec:application_generic}
Minimal geodesics on higher genus surfaces and on the two-torus $T^2$,  i.e. geodesics minimizing the length between any two of its points on the universal covering, were first systematically studied by Morse \cite{Morse24} and Hedlund \cite{Hedlund}. 
Bolotin and Rabinowitz  \cite{BolotinRabinowitz} obtained results about the existence of certain families of  homoclinic and heteroclinic geodesics on $T^2$ that shadow minimal heteroclinics, using a renormalized length functional. As we will explain below, one can obtain from their work (more specifically \cite[Theorem 4.2]{BolotinRabinowitz}) that under certain assumptions there is 
a (non-closed) geodesic in the universal cover which, after applying a family of covering transformations, provides a family of geodesics of a certain intersection pattern, similar to our ribbons. 
 Furthermore, an analogous argument as above in the proof of Theorem \ref{thm2} yields subsequently four closed curves that form a ribbon. 
 The assumption in the theorem above is satisfied for a $C^{\infty}$ generic metric. Hence one can derive the following 


\begin{theorem}\label{thm:generic}
In the space of metrics on $T^2$ with positive topological entropy equipped with the $C^\infty$ topology there is a co-meager set $\mathcal{S}$ such that any $g\in \mathcal{S}$ has robust topological entropy. 
\end{theorem}

We now discuss the result in \cite{BolotinRabinowitz} and its relation to the existence of curves that form a ribbon. 
We keep mainly the notations in \cite{BolotinRabinowitz}.

Assume that the metric $g$ on $T^2$ is not flat. Then there is a simple free homotopy class of closed curves $\alpha$ in $T^2$ and two (possibly identical) minimal geodesics $v_-$ and $v_+$ of class $\alpha$ that from an annulus $A\subset T^2$ that contains no minimal closed geodesics of class $\alpha$ in its interior, see \cite{GlasmachersKnieper}.  Let $S \subset \widetilde{T^2}$ the strip that is the preimage of $A$ under the covering map. 
Let $\tau:\widetilde{T^2} \to \widetilde{T^2}$ be the translation corresponding to $\alpha$. Let $\sigma:\widetilde{T^2} \to \widetilde{T^2}$ be a translation corresponding to a simple homotopy class such that, if $v_- = v_+$ then $\sigma(\widetilde{v}_-) = \widetilde{v}_+$, and otherwise $\sigma(\widetilde{v}_-)$ lies on the left of $\widetilde{v}_+$. 
Let $u_i, i\in \Z$,  be the lifts in $S$ of a shortest geodesic $u$ connecting $v_-$ and $v_+$, with $u_i = \tau^i u_0$. 
Consider the space $\alpha_i$ of (rectifiable) curves $x:[0,1] \to S$ with no constant pieces and such that $x(0) \in u_i$, $x(1) \in u_{i+1}$, and let 
$$\Pi = \{y = (x_i)_{i\in \Z} \, | x_i \in \alpha_i, x_i(1) = x_{i+1}(0)\} \subset \prod_{i\in \Z} \alpha_i.$$ Define the \textit{renormalized length functional} on $\Pi$ as 
$$J(y) := \sum_{i\in \Z} (\mathcal{L}(x_i) - c),$$
whenever the series is convergent, otherwise $J(y) := +\infty$. 
It is explained in \cite{BolotinRabinowitz} that $J$ can be extended to paths $y:\R \to \widetilde{T^2}$, and in particular to $y:\R \to \widetilde{T^2}$ that are negative asymptotic to $\sigma^i(\widetilde{v}_-)$ and positive asymptotic to $\sigma^j(\widetilde{v}_+)$ for some $i,j \in \Z$. 
One defines a \textit{barrier function} $B_-^+$ on $S$ by 
$$
B_-^+(q) := \inf\{J(y) \, | \, y: \R \to S \text{ goes through } q \text{ and is asympt. to }v{\mp} \text{ as } t\to \mp \infty\}.$$
One shows that $B_{-}^+$ is finite, and that the set of minimum points conists of $\widetilde{v}_- \cup \widetilde{v}_+$ and the set of  minimal heteroclinics from $\widetilde{v}_-$ to $\widetilde{v}_+$, which is moreover non-empty. For minimal heteroclinics $h:\R \to S$ one has $B_-^+(h(t)) = J(h)$. Here \textit{minimal heteroclinics} from $\widetilde{v}_-$ to $\widetilde{v}_+$ are globally minimizing geodesics $h(t)$  that are asymptotic to  $\widetilde{v}_{\mp}$ as $t\to \mp \infty$. 
Parts of the Theorems 4.1. and Theorem 4.2. in \cite{BolotinRabinowitz} can be formulated as follows
\begin{theorem}\cite{BolotinRabinowitz}\label{thm:BolRab}
Assume that $B_-^+$ is non-constant, and let $l\in \N$. Then there is a heteroclinic  $\widetilde{\gamma}: \R \to \widetilde{T^2}$ from $\widetilde{v}_-$ to $\sigma^l(\widetilde{v}_+)$, and minimal heteroclinics $h_i, i=0, \ldots, l$ from $\widetilde{v}_-$ to $\widetilde{v}_+$ such that $\widetilde{\gamma}$ shadows $\sigma^i(h_i), i=0, \ldots, l$.  
\end{theorem}   
For the precise definition of shadowing, which is not important for our considerations, we refer to \cite{BolotinRabinowitz}. The time intervals in $\R$ for which $\widetilde{\gamma}$ shadows  $\sigma^i(h_i)$ might be very far apart from each other, and so are the two time intervals in which $\widetilde{\gamma}$ is close to $\widetilde{v}_-$ resp. $\sigma^l(\widetilde{v}_+)$. Furthermore, the heteroclinics $\widetilde\gamma$ constructed in the theorem are embedded. 
The assumptions that $B_-^+$ is non-constant is equivalent to the assumption that there is no foliation of $S$ by minimal heteroclinics from $\widetilde{v}_-$ to $\widetilde{v}_+$.

We now apply Theorem \ref{thm:BolRab} and 
observe that it provides geodesics $\widetilde{\gamma}:\R \to \widetilde{T^2}$ such that together with certain translates, an  intersection pattern similar a ribbon appears, which we call ribbon$^*$, and we proceed with the definition of this property, see Figure~\ref{fig:F*} for an illustration. 
In the following we say that two distinct  geodesics $\eta_1$ and $\eta_2$ in $\widetilde{T^2}$ intersect positively (resp. negatively) at $\eta_1(t_1) = \eta_2(t_2)$ if the orientation given by the tangent vectors $(\eta'_1(t_1), \eta'_2(t_2))$ coincide with the same (resp.  opposite) orientation of that of $T^2$.  

\begin{definition}\label{def:F*}
Let $\gamma$ be (non-closed) geodesic in $T^2$, and $\widetilde{\gamma}$ be a lift to $\widetilde{T^2}$. Assume $\widetilde\gamma$ is embedded. 
We say that $\gamma:\R \to T^2$ (or $\widetilde{\gamma}$), and a family of five covering transformations $\theta_1, \theta_2, \theta_3, \theta_4, \mathcal{T}: \widetilde{T^2} \to \widetilde{T^2}$ form a ribbon$^*$ if 
  for some parameters $s^i_j,  t^i_j \in \R, s^i_j < t^i_j, \, j=1,4; \, i\in \Z$ and  $u^i_j,  v^i_j \in \R, u^i_j < v^i_j, \, j=2,3; \, i\in \Z$, the lifts 
 $\widetilde{\gamma}^i_j = \mathcal{T}^i \circ \theta_j (\widetilde{\gamma}), \, j=1, \ldots, 4; \, i\in \Z$ satisfy, for all $i\in \Z$, the following: 

\begin{enumerate}\setcounter{enumi}{-1}
\item
\begin{itemize} 
\item $\widetilde{\gamma}^i_1$ and $\widetilde{\gamma}^{i+1}_1$ intersect negatively in $\widetilde{\gamma}^i_1(t^i_1) = \widetilde{\gamma}^{i+1}_1(s^{i+1}_1)$, 
\item  $\widetilde{\gamma}^i_4$ and $\widetilde{\gamma}^{i+1}_4$ intersect positively  in $\widetilde{\gamma}^i_4(t^i_4) = \widetilde{\gamma}^{i+1}_4(s^{i+1}_4)$. 
\end{itemize}

\item With $\eta^i_j := \widetilde{\gamma}^i_1|_{[s^i_j,t^i_j]}, \, j=1,4$, the piecewise geodesic 
$\eta_1:= \cdots \eta^{-1}_1 \eta^0_1 \eta^1_1 \cdots$ is on the left of the piecewise geodesic $\eta_4:= \cdots \eta^{-1}_4 \eta^0_4 \eta^1_4 \cdots$. 
\item 
\begin{itemize}
\item $\eta^i_2 := \widetilde{\gamma}^i_2|_{[u^i_2,v^i_2]}$ intersect $\eta_1$ only at the endpoints of $\eta^i_2$, first positively and then negatively, and does not intersect  $\eta_4$.
\item  $\eta^i_3 := \widetilde{\gamma}^i_3|_{[u^i_3,v^i_3]}$ intersect $\eta_4$ only at the endpoints of $\eta^i_3$, first negatively and then positively, and does not intersect $\eta_1$.
\end{itemize}
\item $\eta^i_2$ and $\eta^i_3$ intersect.
\item For all $i,j\in \Z$ with $i\neq j$,  $\eta^i_2$ and $\eta^j_3$ are disjoint,  $\eta^i_2$ and $\eta^j_2$ are disjoint, and  $\eta^i_3$ and $\eta^j_3$ are disjoint. 
\end{enumerate}
	\flushright$\triangle$
\end{definition}

It is now easy to see, and we leave it to the reader to check, see also Figure~\ref{fig:F*}, that if  $\widetilde{\gamma}$ is a heteroclinic from $\widetilde{v}_-$ to $\sigma^{4}(\widetilde{v}_+)$ obtained from Theorem \ref{thm:BolRab}, then for suitably chosen $n_2, n_3,n_4,n_5 \in \N$, $n_3< n_2< n_5< n_4$, the geodesic $\widetilde{\gamma}$ together with the shifts $\theta_1 = \id$, $\theta_2 = \sigma^1 \circ \tau^{n_2}$, $\theta_3 = \sigma^{-2}\circ \tau^{n_3}$, $\theta_4 = \sigma^{-1} \circ \tau^{n_4}$, $\mathcal{T} = \sigma^3 \circ \tau^{n_5} $ form a ribbon$^*$. Hence one can conclude:   

\begin{proposition}
Assume that $B_{-}^{+}$ is non-constant. Then there is a geodesic $\gamma$ and deck transformations $\theta_1, \theta_2, \theta_3, \theta_4, \mathcal{T}$ on $\widetilde{T^2}$ that form a ribbon$^*$. 
\end{proposition}

\begin{figure}
\begin{tikzpicture}[scale=0.5]
\begin{scope}\clip (-1,3) rectangle (26,20);
    \coordinate (A) at (-30,8.2);
\coordinate (B) at (2,8.2);
\coordinate (C) at (10,15.8);
\coordinate (D) at (50,15.8);

\draw[dashed] (A) -- (B);
\draw[-] ($(A)-(0,0.2)$)-- ($(B)-(0,0.2)$);

\path[name path=p1] (B) -- (C) node[midway, above left] {\Large $\widetilde\gamma^i_1$};
\draw[use path=p1,dashed];
\draw[dashed] (C) -- (D);

\draw[dashed] ($(A) + (9,2)$) -- ($(B) + (9,2)$);
\draw[dashed] ($(B)+ (9,2)$) -- ($(C)+ (9,2)$) node[pos=1, above left] {\Large $\widetilde\gamma^i_2$};
\draw[dashed] ($(C)+ (9,2)$) -- ($(D)+ (9,2)$);

\draw[dashed] ($(A) + (1,-4)$) -- ($(B) + (1,-4)$);
\draw[dashed] ($(B)+ (1,-4)$) -- ($(C)+ (1,-4)$);
\draw[dashed] ($(C)+ (1,-4)$) -- ($(D)+ (1,-4)$)  node[pos=0.35, above left] {\Large $\widetilde\gamma^i_3$};

\draw[dashed] ($(A) + (10,-2)$) -- ($(B) + (10,-2)$)node[pos =0.9, below] {\Large $\widetilde\gamma^i_4$};;
\path[name path=p3] ($(B)+ (10,-2)$) -- ($(C)+ (10,-2)$) ;
\draw[use path = p3, dashed]; 
\draw[dashed] ($(C)+ (10,-2)$) -- ($(D)+ (10,-2)$);
\draw[-] ($(C)+ (10,-2)+(0,0.2)$) -- ($(D)+ (10,-2)+(0,0.2)$);

\path[name path=p2]($(A) + (16,6)$) -- ($(B) + (16,6)$)node[midway, above left] {\Large $\widetilde\gamma^{i+1}_1$};
\draw[use path=p2,dashed] ;
\draw[dashed] ($(B)+ (16,6)$) -- ($(C)+ (16,6)$);
\draw[dashed] ($(C)+ (16,6)$) -- ($(D)+ (16,6)$);
\draw[-] ($(B)+ (16,6)-(0,0.2)$) -- ($(C)+ (16,6)-(0,0.2)$);
\draw[-] ($(C)+ (16,6)-(0,0.2)$) -- ($(D)+ (16,6)-(0,0.2)$);

\draw[dashed] ($(A) + (-6,-8)$) -- ($(B) + (-6,-8)$);
\draw[dashed] ($(B)+ (-6,-8)$) -- ($(C)+ (-6,-8)$);
\draw[-] ($(B)+ (-6,-8)+(0,0.2)$) -- ($(C)+ (-6,-8)+(0,0.2)$);
\path[name path = p4] ($(C)+ (-6,-8)$) -- ($(D)+ (-6,-8)$)node[midway, above left] {\Large $\widetilde\gamma^{i-1}_4$};
\draw[use path = p4, dashed];

 \path[name intersections={of={p1} and {p2}, by={I12}}];
\draw[-] ($(B)-(0,0.2)$) -- ($(I12)-(0,0.2)$);
\draw[-] ($(I12)-(0,0.2)$) -- ($(B) + (16,6) -(0,0.2)$);

 \path[name intersections={of={p3} and {p4}, by={I34}}];
 \draw[-] ($(C) + (-6,-8)+(0,0.2)$) -- ($(I34)+(0,0.2)$);
\draw[-] ($(I34)+(0,0.2)$) -- ($(C) + (10,-2)+(0,0.2)$);

\end{scope}
\end{tikzpicture}
\caption{Schematic illustration of a heteroclinic from Theorem \ref{thm:BolRab} with shifts that form a ribbon$^*$. In non-horizontal parts there is a shadowing of minimal heteroclinics. Solid lines illustrate the course of $\eta_1$ and $\eta_4$ as given in Definition~\ref{def:F*}.}
 \label{fig:F*}
\end{figure}
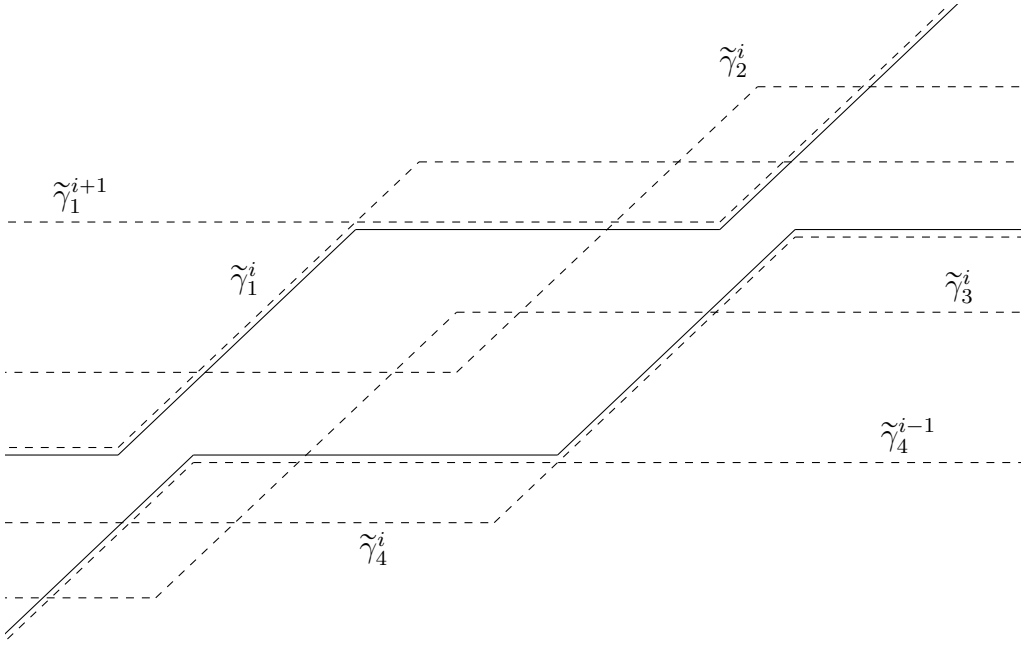

An analogous argument as in the proof of Theorem \ref{thm2} yields
\begin{proposition}
If there is a geodesic $\gamma$ and shifts $\theta_1, \theta_2, \theta_3, \theta_4, \mathcal{T}$ that form a ribbon$^*$,  
then there is $\varepsilon>0$ and four closed geodesics $\tau_1, \ldots, \tau_4$ that form an $\varepsilon$-ribbon. 
\end{proposition}

\begin{proof}
The piecewise geodesics $\eta_1$ and $\eta_4$ form an infinite strip with piecewise geodesic boundary with outer angles $>\pi$ at the non-smooth points. Define for all $i\in \Z$ the non-empty sets $U^i_0 = R(\eta_1) \cap L(\gamma^i_2)$, $U^i_1 =  L(\eta_4) \cap R(\gamma^i_3)$.  By item $(3)$ of the ribbon$^*$, $U^i_0 \cap U^i_1 \neq \emptyset$. One can define for any binary bi-infinite sequence $\mathfrak{a}$ as in the proof of Theorem \ref{thm1} (using the shift $\mathcal{T}$) sets $D(\mathfrak{a})$ which, by item $(4)$ of the ribbon$^*$, are infinite strips with piecewise geodesic boundary with outer angles $>\pi$ at the non-smooth points.  Now one can proceed as in proof of Theorem \ref{thm2} to obtain, for some $\varepsilon>0$, closed geodesics $\tau_1, \ldots, \tau_4$ that form an $\varepsilon$-ribbon. 
\end{proof}
\begin{corollary}
Let $g$ be a metric on $T^2$. If there is a geodesic $\gamma$ and shifts $\theta_1, \theta_2, \theta_3, \theta_4, \mathcal{T}$ that form a ribbon$^*$, then there is $\delta>0$ such that for all $g'$ with $d_{C^0}(g',g)< \delta$ the geodesic flow of $g'$ has positive topological entropy. 
\end{corollary}

  Note that for bumpy metrics closed minimal geodesics are hyperbolic, and note that bumpy is a $C^{\infty}$ generic condition. In case $B_-^+$ is constant on $S$, then $S$ is foliated by minimal heteroclinics from $v_-$ to $v_+$, in other words the unstable manifold of $v_-$ and the stable manifold of $v_+$ intersect, but not transversely. One can perturb in a neighborhood of any point of such a heteroclinic $h_{vol}$, as was shown by Donnay  \cite{Donnay} ($C^2$ perturbation) and Petroll  \cite{Petroll} ($C^{\infty}$ perturbation)  (for a sketch of the proof see also \cite{BurnsWeiss}),   such that $h_{vol}$ becomes a transverse heteroclinic connection from $v_-$ to $v_+$.  If such a perturbation is sufficiently small, $v_-$ and $v_+$ stay to be adjacent (minimal) geodesics, where now $B_-^+$ is non-constant. So the assumptions in Theorem \ref{thm:BolRab} hold $C^{\infty}$ generically, which assures that Theorem \ref{thm:generic} holds.

\section{Robustness by retractable neck on general manifolds}\label{sec:retractibleNeck}

Consider a closed Riemannian manifold $(M,g)$ that is not necessarily a torus. We assume that there exist nested nonempty open sets $U\subseteq V_1\subseteq V_2\subseteq W$ whose closure have smooth boundary such that \[U\subseteq \overline U\subseteq V_1\subseteq\overline V_1\subseteq V_2\subseteq\overline V_2\subseteq W\] and such that there is a retraction $\rho:M\backslash U\to M\backslash W$ that is homotopic to the identity relative $M\backslash W$ through the homotopy $\rho_t$. 
We denote the two numbers:
\begin{align*}
	d_1 &= \max_{x\in\partial V_2} l_g(\rho_t(x)),\\
	d_2 &= dist_g(\partial V_1, \partial V_2).
\end{align*}
We interpret this setup as follows: $U$ is a head, that we intend to cut off. The set $W\backslash U$ is a neck that is further divided into lower neck $W\backslash V_2$, middle neck $V_2\backslash V_1$ and upper neck $V_1\backslash U$. The numbers $d_i$ are the length of the lower neck $d_1$ and the length of the middle neck $d_2$, measured in a way that suits later proofs.

\begin{assumption}[Retractable neck and entropic body] If the following statements are true, we say that the ''neck'' $W\backslash U$ is $(c,k)$-\emph{retractable} for a number $c\in(0,1)$ and $k\geq 3$.
\begin{itemize} 
	\item The retraction $\rho$ is a contraction: for any curve $\gamma:I\to M\backslash U$ we have that $l_g(\rho\circ\gamma)\leq l_g(\gamma)$. 
	\item The retraction $\rho$ is a proper contraction in the middle and upper neck: for any curve $\gamma:I\to V_2\backslash U$ we have $l_g(\rho\circ\gamma)< c\, l_g(\gamma)$.
	\item The lower neck is substantially shorter than the middle neck: 
	\begin{align}
		\frac{d_1}{d_2} &< \frac{1-c}{k} \label{LowerToUpper}
	\end{align}
\end{itemize}	
	If the following statement is true, we say that $M$ has an \emph{entropic body}. 
\begin{itemize}
	\item There is a subset $\Pcal\subseteq \tiilde\pi_1(M\backslash U)\backslash \iota \tiilde\pi_1(\partial U)$ of the free homotopy classes of $M\backslash U$ not homotopic to curves in $\partial U$, whose elements are mutually coprime such that the subsets
	$$\Pcal_g(T)=\{\alpha\in\Pcal\mid \exists \gamma\in\alpha : l(\gamma)\leq T\}$$ 
	 grow exponentially: $\Gamma_T(\#\Pcal_g(T)) >0$.
\end{itemize}
\hfill$\triangle$
\end{assumption}


\begin{remark}
	Note that even though $\Gamma_T(\#\Pcal_g(T))$ depends on the metric, the positivity $\Gamma_T(\#\Pcal_g(T))>0$ is a purely algebraic statement about the group growth of the free homotopy classes seen as $\tiilde\pi_1=\pi_1/conj$. In particular, if $\pi_1$ has a subgroup isomorphic to $\ZZ*\ZZ$, then the assumption is satisfied.
\end{remark}

\begin{theorem}\label{thm:entropyFromNeck}
	Let the closed Riemannian manifold $(M,g_0)$ have a $(c,k)$-retractable neck and an entropic body. Then, for $C$ with $1<C<\frac k{2+(k-2)c}$, $g_0$ has $C$-robust positive topological entropy and for $g\sim_Cg_0$ we have
	$$h_{\rm top}(\varphi_g^t)\geq \frac1{\sqrt C} \Gamma(\#\Pcal_{g_0}(T)).$$ 
\end{theorem}

\begin{example}\label{ex:addHead}
	Let $M^n$ be a manifold such that there exists a Riemannian metric whose topological entropy vanishes. Fix $c$ and $k$. Consider a submanifold $L$ such that $M\backslash L$ is an entropic body. We identify a tubular neighborhood of $L$ with the normal disk bundle $DL$ with radial coordinate $r\in[0,3+\frac k{2+(k-2)c}]$. 
	We define the neck 
	$$U=\{r<1\},V_1=\{r<2\},V_2=\left\{r<2+\frac k{2+(k-2)c}\right\},W=\left\{r<3+\frac k{2+(k-2)c}\right\}$$ with the retraction defined by 
	$$\rho_s|_{M\backslash W}=id,\quad \rho_s(r,x)=\left(\min\left\{r+\left(2+\frac k{2+(k-2)c}\right)s,3+\frac k{2+(k-2)c}\right\},x\right).$$	
	Endow the neck with a metric $g$ such that 
	$$g|_{r\in(1,3+\frac k{2+(k-2)c})}=f(r)g_{SL}+dr^2,$$ 
	where $g_{SL}$ is a metric of the normal sphere bundle over $L$ and $f$ is a function in $r$ with $f(r)\geq f(3+\frac k{2+(k-2)c})$ for all $r\in(1,3+\frac k{2+(k-2)c})$ and $c\sqrt{f(r)}\geq \sqrt{f(3+\frac k{2+(k-2)c})}$ for all $r\in(1,2+\frac k{2+(k-2)c})$. With this metric, the neck has the $(c,k)$-retractable neck property. Thus, any extension of $g$ to $M$ satisfies the assumptions of the theorem.
	
	In dimension $2$, $L$ is discrete. For $S^2$ three points and for $T^2$ one point suffice.
	In dimension $3$, $L$ is a link. For $S^3$ the unlink and for $T^3$ the unknot suffice.
	\hfill$\triangle$
\end{example}

The following Lemma is the $C^0$-robust property that we derive from a retractable neck. 

\begin{lemma}\label{minimizer}
	Let $(M,g_0)$ have a $(c,k)$-retractable neck. Let $g$ be a metric with $g \sim_C g_0$ for some number $1 < C < \frac k{2+(k-2)c}$. Let $\alpha$ be a homotopy class of a curve in $M\backslash U$ that is not homotopic ot a curve in $\partial U$. Then, any $g$-length minimizer of $\alpha$ has image in $M\backslash V_1$. 
\end{lemma}
\begin{proof}

	Let $\gamma\in\alpha$. Assume $\exists T: \gamma(T)\in V_1\backslash U$. We claim that $\gamma$ is not a length minimizer of $\alpha$. We prove this by explicitly constructing a shorter curve homotopic to $\gamma$.
	
	There is a maximal connected neighborhood $I\subseteq S^1$ of $T$ such that $\gamma(I)\subseteq V_1\backslash U$. Since $\gamma \notin \iota\pi_1\partial U$, the interval $I$ is not the entire circle. Because of maximality of $I=[t_1,t_2]$ the end points lie in the boundary $\gamma(t_1),\gamma(t_2)\in\partial V_1$. This implies that that $l_{g_0}(\gamma|_I)\geq 2 d_2$ by definition of $d_2$.
	
	We define the homotopy $\gamma_s(t)$, $s\in[0,1]$ as the concatenation 
	$$\gamma_s(t)= \gamma|_{S^1\backslash I}\circ \rho|_{[0,s]}(\gamma(t_1)) \circ \rho_s(\gamma|_I) \circ \overline{\rho|_{[0,s]}}(\gamma(t_2)).$$
	Obviously $\gamma_0 \sim \gamma_1$. It is elementary to verify that the condition $C < \frac k{2+(k-2)c}$ implies that $\frac{1-cC}{2C}>\frac{1-c}k$. We compute
	
	\begin{align*}
		l_g(\gamma_0)-l_g(\gamma_1) &= l_g(\gamma|_I) - l_g(\rho\circ\gamma|_I) - l_g(\rho_s(\gamma(t_1))) - l_g(\overline{\rho_s}(\gamma(t_2)))\\
		&\geq \frac1{\sqrt{C}} l_{g_0}(\gamma|_I) - {\sqrt{C}} l_{g_0}(\rho\circ\gamma|_I) - {\sqrt{C}} l_{g_0}(\rho_s(\gamma(t_1))) - {\sqrt{C}} l_{g_0}(\overline{\rho_s}(\gamma(t_2)))\\
		&> {\sqrt{C}}\left( \frac1{C} l_{g_0}(\gamma|_I) - c l_{g_0}(\gamma|_I) - l_{g_0}(\rho_s(\gamma(t_1))) - l_{g_0}(\overline{\rho_s}(\gamma(t_2))) \right) \\
		&\geq {\sqrt{C}}\left( \left(\frac1{C}-c\right) d_2 - 2d_1 \right)  =  {\sqrt{C}} 2 \left( \left(\frac{1-cC}{2C}\right) d_2 - d_1 \right)\\
		&> 2 {\sqrt{C}} \left(\frac {1-c}k d_2 -d_1\right)\\
		&> 0.
	\end{align*}
	We conclude that $\gamma=\gamma_0$ is not a length minimizer.
\end{proof}

\begin{proof}[Proof of Theorem~\ref{thm:entropyFromNeck}]
	Let $g\sim_C g_0$ and $\alpha\in\Pcal$. Since $\alpha$ is non-contractable, the infimal length of the homotopy class is positive $l(\alpha):=\inf\{l(\gamma)\mid \gamma\in\alpha\}> 0$. Let $\gamma_k:S^1\to M\backslash U$ be a sequence of smooth loops parametrized by constant speed with $l(\gamma_k)\to l(\alpha)$.  Since $|\dot\gamma|\to l(\alpha)$ and since $M\backslash U$ is compact, we can apply Arzela--Ascoli and find a subsequence that converges to a curve $\gamma_{\alpha,g}$ which satisfies $l(\gamma_{\alpha,g})\leq l(\alpha)$ by lower semi-continuity of the length functional. By minimality of $l(\alpha)$ this implies $l(\gamma_{\alpha,g})=l(\alpha)$. Thus, $\gamma_{\alpha,g}$ is a length minimizer. Lemma~\ref{minimizer} tells us that the image of a length minimizer is contained in $M\backslash V_1$, which is in the interior of $M\backslash U$. We conclude that $\gamma_{\alpha,g}$ is a geodesic.
	
	Thus, for every $g\sim_C g_0$ and $\alpha\in\Pcal$ there is a length minimising geodesic $\gamma_{\alpha,g}:l_g(\alpha)S^1\to M\backslash U$, which we parametrize from now on by arc length for convenience. Note that $\gamma_{\alpha,g}$ lifts to a periodic orbit $(\gamma_{\alpha,g},\dot\gamma_{\alpha,g})$ of $\varphi_g^t$ of period $l_g(\alpha)$.
	
	The relation $l_g\leq \sqrt Cl_{g_0}$ implies that $\frac 1{\sqrt C} l_g(\alpha)\leq l_{g_0}(\alpha)$. Thus, 
	$$\{\alpha\in\Pcal\mid l_{g_0}(\alpha)\leq T\}\subseteq \{\alpha\in\Pcal\mid l_g(\alpha)\leq \sqrt{C}T\}$$
	and consequently the sets $\tiilde\Pcal_{g}(T)=\{\gamma_{\alpha,g}\mid l_g(\gamma_{\alpha,g})<T\}$ satisfy
	$$\Gamma(\#\tiilde\Pcal_{g}(T))\geq \frac1{\sqrt{C}} \Gamma(\#\Pcal_{g_0}(T)).$$
	
The desired statement now follows from Lemma~\ref{lem:nonhomotopicGeodesicsToEntropy}.
\end{proof}

\begin{proof}[Proof of Theorem~\ref{thm:EntropyDenseAndBig}]

Let $(Q,g)$ be the $k\geq 2$ dimensional Riemannian manifold and let $e>0$ be arbitrary. We search metrics $g(s)\in\mathfrak{G}^e(Q)$ with $d_{C^0}(g,g(s))<s$ that has a $(c(s),k(s))$-retractable neck, where $\lim_{s\to 0}(c(s),k(s))=(1,3)$ and $\lim_{s\to \infty}(c(s),k(s))=(0,\infty)$. Theorem~\ref{thm:entropyFromNeck} then implies the statement qualitatively. The formula in Theorem~\ref{thm:EntropyDenseAndBig} comes from the specific construction.

We briefly outline the argument: We prepare a small disk in which all geodesics considered will be contained. Then, we construct some heads inside such that the homotopy classes of curves in the disk minus the heads have positive algebraic growth. The growth of homotopy classes filtered by length will be at least the algebraic growth divided by length of longest generator. Then, we scale the entire construction down inside the disk, leaving the algebraic growth invariant but reducing the length of longest generator. This way, we find arbitrarily large entropy. This construction can be done by a $C^0$-small perturbation of the metric that is parametric in $s$.

{\bf Step 1: Choice and manipulation of a small disk}
For an arbitrary point $p$, we choose a nearby metric $g_1=g_1(s)$ that is slightly reshaped around $p$: We flatten a small disk surrounded by a thin cylindrical annulus. 
To quantify small, we choose $0<\epsilon_1$ and $0<\epsilon_2\ll\epsilon_1$ in dependence of $s$: $\epsilon_1(s)$ is a continuous function that is linear for small $s$ and constant $\ll1$ for $s>s_0$ for some $s_0\ll 1$. 

We choose then $g_1$ such that 
\begin{itemize}
    \item $g_1\equiv g_0$ on $Q\backslash B_{4\epsilon_1}(p)$,
    \item $g_1 = f(r)g_{S^{k-1}}+dr^2$ on $B_{3\epsilon_1}(p)$, where $r$ is a radial coordinate,
    \item $f(r)\equiv 4$ on the annulus $r\in(2\epsilon_1-\epsilon_2,2\epsilon_1+\epsilon_2)$,
    \item $f(r)=r^2$  on $B_{\epsilon_1}(p)$,\\
    \item $d_{C^0}(g_0,g_1)< s/2$,
\end{itemize}
where $g_{S^{k-1}}$ is the round metric on the euclidean sphere and where the radii of balls are measured with respect to $g_1$. 

Note that choosing $g_1$ is very easy starting from coordinates that are orthonormal on $T_pQ$ since we allow $C^0$-small perturbations. It would be impossible for $C^2$-small perturbations, as curvature would be an obstruction. The deformation around the annulus is a small deviation from the flat metric as long as $\epsilon_2$ is small in comparison to $\epsilon_1$. This is a manifestation of the fact that all changes of $f$ by quantities that are small with respect to $f$ are small.

The condition on the annulus is to ensure that the disk $B_{\epsilon_1}(p)$ is surrounded by a totally geodesic codimension 1 sphere, which will help to contain minimizing curves within the disk in its interior.

{\bf Step 2: Choice of heads}
Let $\iota:L\hookrightarrow B_{\epsilon_1}(0)$ be an embedded codimension $2$ submanifold: If $k=2$, then we choose $L$ to be three points, if $k>2$ then let $L= L_1\cup L_2$ have two unknotted components with $L_i\cong S^1\times S^{k-3}$. Note that in both cases the group growth $\Gamma(\pi_1(B_{\epsilon_1}(p)\backslash L,p))=:\Gamma$ is positive. There is a length $\lambda(L)$ such that there are generators of $\pi_1(B_{\epsilon_1}(p)\backslash L,p)$ of length at most $\lambda(\iota L)$. To demonstrate the future argument, denote by $\Pcal_{g_1}(\iota,T)$ the set of free homotopy classes of loops in $B_{\epsilon_1}(p)\backslash \iota L$ that are represented by a loop of length $\leq T$ and that do neither retract onto $\iota L$ nor to $\partial B_{\epsilon_1}(p)$. Then we have
\[\Gamma_T(\#\Pcal_{g_1}(\iota,T)) \geq \Gamma / \lambda\]
as for each free loop we find a (longer) representing based loop and the conjugacy classes of the fundamental group grow as fast as the fundamental group.
If we postcompose the embedding $\iota$ with a dilation by a factor $t$, the algebraic growth $\Gamma$ will obviously not change but the group will be generated by loops of length $\lambda(t\iota L)=t\lambda(\iota L)$, implying that $\Gamma_T(\#\Pcal_{g_1}(t\iota,T))\to\infty$ as $t\to 0$. 

{\bf Step 3: Shaping the necks}
The shape of the necks is determined similar to the shape of the base disk. We choose in dependence of $s$ the shape parameters of the neck $(c(s),k(s))$ such that $c(s)<e^{-s/8}$ and $k=3+s$. Further, we choose new and even smaller $0<\epsilon_3$ and $0<k\epsilon_4\ll\epsilon_3$. We choose a nearby metric $g_2=g_2(s)$ that is flattened in an $\epsilon_3$-tube around $L$, except for an $\epsilon_4$ wide annulus which imitates Example~\ref{ex:addHead}. More precisely,
\begin{itemize}
    \item $g_2\equiv g_1$ outside the tubular neighborhood $V_{3\epsilon_3}(N)$,\\
    \item $g_2\equiv f(r)g_{SL}+dr^2$ inside $V_{2\epsilon_3}(N)$,\\
    \item $f(r)\geq f(\epsilon_3+(3+\frac k{2+(k-2)c})\epsilon_4)$ for all $r\in(\epsilon_3+\epsilon_4,\epsilon_3+(3+\frac 3{1-c})\epsilon_4)$,\\
    \item $c\sqrt{f(r)}\geq \sqrt{f(\epsilon_3+(3+\frac k{2+(k-2)c})\epsilon_4)}$ for all $r\in(\epsilon_3+\epsilon_4,\epsilon_3+(2+\frac k{2+(k-2)c})\epsilon_4)$,
    \item $d_{C^0}(g_2,g_1)<s/2$,
\end{itemize}
where tubular neighborhoods are taken with respect to the metric $g_2$ and where $g_{SL}$ is the metric of normal sphere bundle over $L$.
Note that if $f(r)$ would equal $r^2$, then the metric would be flat and the assumption would be similar to the choice of the small disk. 

Note that the crucial feature of our choice of $c(s)$ is that $\log(1/c^2)<s/2$ because the fourth point forces us to deviate by a factor $c$ from the cylindrical metric, which forces $d_{C^0}(g_2,g_1)>\log(1/c^2)$ and the fith point demands $d_{C^0}(g_2,g_1)<s/2$. 

This new metric has a retractable neck with sets 
\[U=\{r<\epsilon_3+\epsilon_4\},V_1=\{r<\epsilon_3+2\epsilon_4\},\]\[V_2=\left\{r<\epsilon_3+(2+\frac k{2+(k-2)c})\epsilon_4\right\},W=\left\{r<\epsilon_3+(3+\frac k{2+(k-2)c})\epsilon_4\right\}.\]

The inclusion $B_{\epsilon_1}(p)\backslash U \hookrightarrow B_{\epsilon_1}(p)\backslash L$ is obviously a homotopy equivalence. The generators of the fundamental group are possibly a bit larger, but $2\lambda(\iota L)$ suffices if $\epsilon_3$ is small enough.

{\bf Step 4: Shrinking for growth}
Now, we employ the dilation by $t$ mentioned in Step 2 for the metrics $g_2(s)$ from Step 3. To be more formal, denote by $\delta_{t}$ the dilation by $t$ in the flat model around $p$ and by $g_t(s)$ the metric which coincides with $g_2(s)$ on $Q\backslash \delta_s(B_{2\epsilon_1-\epsilon_2})$ and with $t^2\delta_{1/t}^*g_2$ in a small neighborhood. Note that $t^2\delta_{1/t}^*g_{\rm Euc}=g_{\rm Euc}$ and that the scaling leaves ratios intact, so $d_{C^0}(g_t(s),g_1)=d_{C^0}(g_2(s),g_1)$ and in total \[d_{C^0}(g_t(s),g_0)<d_{C^0}(g_t(s),g_1(s))+d_{C^0}(g_1(s),g_0)<s.\]

Let $\rho$ be a free homotopy class of loops in the disk with the scaled heads $U_s=\delta_s U$ removed $B_{2\epsilon_1}(p)\backslash U_s$, which neither retracts to $\partial U$ nor to $\partial B_{2\epsilon_1}$. Choosing a length infimizing sequence, we find by Arzela--Ascoli up to choice of subsequence a limit loop $\gamma$ for $\rho$. This minimizer cannot touch $\partial U$ by construction of a retractable neck. Nor can it touch $\partial B_{2\epsilon_1}$ as otherwise it would be tangent to a geodesic in the geodesic foliation of $\partial B_{2\epsilon_1}$ that comes from the cylindric metric on the annulus $r\in(2\epsilon_1-\epsilon_2,2\epsilon_1+\epsilon_2)$ and thus be a geodesic belonging to that foliation, contradicting our assumption on its homotopy class. Thus, each class in $\Pcal_{g_1}(si,T)$ from Step 2 is represented by a geodesic.

As noted in Step 2, the fundamental group $\pi_1(B_{\epsilon_1}(p)\backslash U_s,p)$ is generated by loops of length $<2t\lambda(\iota L)$. Thus, we may choose for each $s$ a $t$ so small that $\Gamma/(2t\lambda(\iota L))>e$, where $e$ is the exponential growth required in the statement of the theorem. To describe the necessary choice in dependence of $s$, note that by Theorem~\ref{thm:entropyFromNeck} for our choices $c(s)<e^{-s/8}$ and $k=3+s$ we obtain the bound that for \[C<\frac{s+3}{2+(s+1)e^{-s/8}}\]
we can expect for $d_{C^0}(g_t(s),g)<C$ that $h_{\rm top}(\varphi^t_g)\geq \frac1{\sqrt C}\Gamma(\# \Pcal_{g_t(s)}(T))$. Thus, in order to enforce $h_{\rm top}(\varphi^t_g)\geq e$ within the $\frac{s+3}{2+(s+1)e^{-s/8}}$-balls, we must have a dilation by at least
\[t\sim \mbox{const} \sqrt{\frac{2+(s+1)e^{-s/8}}{s+3}},\]
where the constant is in dependence of $\Gamma$ and $\lambda$ for a specific value. This gives us the required growth of minimizing geodesics which concludes the argument. 
\end{proof}

\begin{proof}[Proof of Corollary~\ref{cor:EntropyDenseAndBig}]
    The first two points in the corollary are immediately clear. For the third point, we start with the quasi-isometric embedding  $\Phi_n:(\R^n, |\cdot|_{\infty}) \to (\overline{\mathfrak{G}}(T^2), d_{\rm RBM})$ from~\cite{JunVukasin}, where the volume is fixed as 1 and the diameter bound is 100. Note that if $\Phi_n$ is quasi-isometric and $\tiilde\Phi_n$ is $d_{C^0}$-close to $\Phi_n$, then also $\tiilde\Phi_n$ is quasi-isometric. So the statement is proved by parametrically performing the above construction. Note that for this only the first step needs to be done parametrically, as from then on the construction is on the small disk which is flat for any starting metric. It is also sufficient to choose one constant but small $s$ and a corresponding constant parameter $t$. As volume and diameter are $C^0$-continuous, the perturbed metrics have volume $1$ after a rescaling close to a factor 1 and the diameter still admits the bound of 101 as stated in our corollary. 
\end{proof}

\appendix

\section{Robustness of non-degenerate length spectrum}\label{appendix:Spectrum}

Here, we prove Proposition~\ref{stabl}. The aim is to use only a low amount of technology.

\begin{remark}Unfortunately, one cannot say anything about the position of the geodesic that is found by the theorem.
\flushright$\triangle$ \end{remark}

Before we start the proof, let us fix a setup: Denote by $\Omega=H^1(S^1,M)$ the Hilbert manifold\footnote{We use this setting to avoid working in the Fr\'echet manifold $C^\infty(S^1,M)$. However, by bootstrapping every geodesic ends up being smooth.} of closed loops in $M$. The non-constant critical points of the energy functional $\Ecal_g:\gamma\mapsto \frac12\int_0^1 g(\dot\gamma,\dot\gamma)\;dt$ are exactly the closed geodesics. The negative gradient flow $\varphi_g^t$ of $\Ecal_g$ has the Palais--Smale property in this space. Denote the sublevel set $\Omega^a_g:=\{\gamma\in\Omega\mid \Ecal_g(\gamma)\leq a\}.$

That $\gamma$ is non-constant and non-degenerate means that the connected component of $\Crit\Ecal_g$ containing $\gamma$ is a circle and Morse--Bott. If all geodesics are non-degenerate, then the energy spectrum is discrete. The following statement describes what happens topologically at a critical energy level.

\begin{proposition}[\cite{B54}, see also\cite{O14}]\label{prop:cycle}
Assume that $c\in(a,b)$ is the only critical value in $[a,b]$. Denote $N_1,\ldots,N_r$ the components of $\Crit(\Ecal)$ with $\Ecal(N_i)=c$ and with indices $\lambda_1,\ldots,\lambda_r$. Assume they are Morse--Bott. Then 
\begin{itemize}
	\item Each manifold $N_i$ carries a well defined vector bundle $\nu^-N_i\subset T\Omega|_{N_i}$ of rank $\lambda_i$ consisting of negative directions of $d^2 \Ecal_g$.
	\item The sublevel set $\Omega^b_g$ retracts onto a space homeomorphic to $\Omega^a_g$ with the disc bundles $D\nu^-N_i$ disjointly attached to $\Omega^a_g$ along their boundaries.
	\item The retraction $r:\Omega^b_g\to\Omega^a_g\bigcup_{\partial D\nu^-N_i} D\nu^-N_i$ can be chosen such that $\Ecal_g\leq \Ecal_g\circ r$ and such that $r|_{N_i}=id$ and $r|_{\Omega^a_g}=id$.
\end{itemize}
\end{proposition}

\begin{remark}
	This proposition gives inductive instructions to build a CW-complex homotopy equivalent to $\Omega$. The building blocks are disk bundles, which are cell complexes. The retraction maps inductively provide the attaching maps. 
\end{remark}

Since we are only interested in the topology, we use the term \emph{topologically non-degenerate} for a curve where the conclusions of Proposition~\ref{prop:cycle} hold.

\begin{definition}\label{def:topologicalNonDegeneracy}
	We assume that $c\neq 0$ is the only critical value in $(a,b)$. Denote $N_i$ the components of $\Crit(\Ecal)$ and assume that they are all isolated circles representing reparametrizations of non-constant closed geodesic $\gamma_i$ with energy $c$.
	
	Then we call $\gamma_i$ \emph{topologically non-degenerate} if there are vector bundles $\nu^-N_i\subseteq T\Omega|_N$ such that the sublevel set $\Omega_g^b$ retracts onto a space homeomorphic to $\Omega_g^a$ with the disc bundles $D\nu^-N_i$ attached to $\Omega_g^a$ along the boundary via a retraction $r$ with $\Ecal_g\leq \Ecal_g\circ r$ and such that $r|_{N_i}=id$ and $r|_{\Omega_g^a}=id$. 	\flushright$\triangle$
\end{definition}

\begin{remark}\label{rem:localization}
    The assumption that the spectral value is isolated is actually too strong for our purpose; It would suffice to demand in Therem~\ref{thm:robustDM} that a topologically non-degenerate $\gamma$ be isolated in the space of loops. The proof below would then work by localizing the gradient flow. One would do this by multiplying the gradient vector field with a bump function around a neighborhood of $N_i$ that is flow-invariant in the intended energy interval, and that separates $\gamma$ from other geodesics. The argument would become much more complicated as the resulting flows only locally transport the respective sub-level sets into each other.
\end{remark}

	We intend to use a minimax principle. We use the following formulation from Klingenberg~\cite{K78}. A \emph{flow-family} $\Acal$ for $\Ecal_g$ is a collection of subsets of $\Omega$ such that $\Ecal_g|_{A}$ is bounded for all $A\in \Acal$ and such that $A\in\Acal$ implies $\varphi_g^t A\in\Acal$ for $t\geq 0$.
\begin{proposition}\label{minimax}
	Let $\Acal$ be a flow-family for $\Ecal_g^t$.	Then 
	$$\inf_{A\in\Acal}\sup_{A} \Ecal_g$$
	is a critical value of $\Ecal_g$.
\end{proposition}

\begin{proof}[Proof of Proposition~\ref{stabl}]
We use Proposition~\ref{prop:cycle} to define a suitable flow-family. For simplicity, assume that there is only one critical component. For Proposition~\ref{stabl} it is enough to consider the case $N_1=N\iso S^1$. The fundamental class of the transverse bundle relative its boundary $[D\nu^-N;\partial D\nu^-N]$ has nonempty intersection with the core $N$ since it has nonempty intersection with any interior point. By extension the same is true for the class $\omega:=[\Omega^a_g\bigcup_{\partial D\nu^-N}D\nu^-N;\Omega^a_g]$. Denote by $r^*\omega$ the set of maps $u:(D\nu^-N;\partial D\nu^-N)\to (\Omega^b_g,\Omega^a_g)$ such that $[r\circ u]=\omega$. Then, the set of images of $u\in r^*\omega$ defines a flow-family.

The minimax value for $r^*\omega$ is the critical value $c$: 
$$\inf_{u\in r^*\omega}\max \Ecal\circ u \geq \inf_{u\in r^*\omega} \max\Ecal\circ r\circ u \geq \Ecal(N)=c.$$
The other inequality is trivial since $\Ecal_g$ restricted to the unstable disk bundle of $N$ has maximum~$c$.

The robustness statement now follows by using the very same retraction $r$ to define a flow family for the perturbed metric $\tiilde g$: Let $\varepsilon>0$ be so small that $c$ is the only critical value in $[(1-3\varepsilon )c,(1+3\varepsilon)c]$. Let $\tiilde g$ be a metric such that $\|v\|^2_{\tiilde g}\in(1-\frac12\varepsilon,1+\frac12\varepsilon)\|v\|^2_{g}$ for all $v$. 
Note that for such $\varepsilon$ small enough the following chain of inclusions holds
$$\Omega^{(1-2\varepsilon)c}_{\tiilde g}\subseteq\Omega^{(1-\varepsilon)c}_{g}\subseteq\Omega^c_g\subseteq\Omega^{(1+\varepsilon)c}_{\tiilde g}\subseteq\Omega^{(1+2\varepsilon)c}_{g}.$$ 
Let $r:\Omega^{(1+2\varepsilon)c}_g\to\Omega^{(1-\varepsilon)c}_g\bigcup_{\partial D\nu^-N} D\nu^-N$ be the retraction constructed with $\varphi_g^t$ and $r^*\omega$ the class described above. Define the subclass $\tiilde\omega\subset r^*\omega$ by restriction of the target space $u:(D\nu^-N;\partial D\nu^-N)\to (\Omega^{(1+\varepsilon)c}_{\tiilde g},\Omega^{(1-2\varepsilon)c}_{\tiilde g})$. The set of images of representatives of $\tiilde \omega$ is a flow-family for $\varphi_{\tiilde g}^t$ since it is defined through sublevelsets of $\Ecal_{\tiilde g}$, and it is nonempty since it contains the $\varphi_g^t$-unstable disk bundle around $N$.
We have 
$$\inf_{u\in\tiilde\omega}\max \Ecal_{\tiilde g}\circ u \geq \inf\max_{u\in\tiilde\omega} (1-\varepsilon)\Ecal_{g} \circ u \geq (1-\varepsilon)c.$$
On the other hand for $u$ the $\varphi^t_g$-unstable disk bundle around $N$ we have $\max \Ecal_{\tiilde g}\circ u\leq (1+\varepsilon)c$. Thus, the minimax principle produces some geodesic $\tiilde \gamma$ of $\Ecal_{\tiilde g}$ with energy $|\Ecal_{\tiilde g}(\tiilde \gamma)-c|\leq \varepsilon c$. 

Note that for any $u$ in the flow-family every path in the image of $u$ is homotopic to a loop in $N$ since the intersection of $u$ and $N$ is nonempty. Thus, also $\tiilde\gamma$ is homotopic to the unperturbed geodesic.
\end{proof}

\bibliographystyle{alpha}

\bibliography{biblio}

\end{document}